\documentclass[12 pkt]{article}
\usepackage{amssymb,amsmath,amsthm}
\usepackage{amscd}
\usepackage{graphicx}
\usepackage{esint} 
\usepackage{color}
\usepackage{soul}
\numberwithin{equation}{section}
\newtheorem{theorem}{Theorem}[section]
\newtheorem{lemma}[theorem]{Lemma}
\newtheorem{corollary}[theorem]{Corollary}
\newtheorem{proposition}[theorem]{Proposition}

\theoremstyle{definition}

\newtheorem{remark}[theorem]{Remark}

\theoremstyle{remark}

\newcommand{\ao}{\vec{A}}
\newcommand{\bo}{\vec{B}}

\newcommand{\bx}{\vec{\mathcal{X}}}
\newcommand{\xo}{\vec{X}}
\newcommand{\yo}{\vec{Y}}
\date{}

\title{\bf  A~unified approach to compatibility theorems on invertible interpolated operators}
\author{I.~Asekritova, N.~Kruglyak and M.~Masty{\l}o}
\begin{document}
\maketitle

\noindent
\renewcommand{\thefootnote}{\fnsymbol{footnote}}
\footnotetext{2010 \emph{Mathematics Subject Classification}: Primary 46B70, Secondary 47A13.} \footnotetext{\emph{Key words and phrases}:
Interpolation functor, the complex interpolation method, the real interpolation method, uniqueness-of-inverses property, spectrum of interpolated operators.}
\footnotetext{The third author was supported by the National Science Centre, Poland, Grant no. 2015/17/B/ST1/00064.}

\begin{abstract}
\noindent
We prove the stability of isomorphisms between Banach spaces  generated by interpolation methods introduced by
Cwikel--Kalton--Milman--Rochberg which includes, as special cases, the real and complex methods up to equivalence
of norms and also the so-called $\pm$ or $G_1$ and $G_2$ methods defined by Peetre and Gustavsson--Peetre.~This
result is used to show the existence of solution of certain operator analytic equation.~A~by-product of these results
is a~more general variant of the Albrecht--M\"uller result which states that the \linebreak interpolated isomorphisms
satisfy uniqueness-of-inverses between interpolation spaces. We show applications for positive operators between
Calder\'on function lattices. We also derive connections between the spectrum of interpolated operators.
\end{abstract}

\vspace{5 mm}

\section{Introduction}

In Banach space theory, operator theory plays a fundamental role. An important part of this theory is the spectral theory which
has applications in many areas of modern analysis and physics. The study of stability properties of interpolated operators
is a~central task in abstract interpolation theory. Motivated by applications in the mentioned areas of analysis, we study
stability  and the local uniqueness-of-inverse properties of interpolated isomorphisms between Banach spaces generated
by some general interpolation methods.

As usual for a given Banach space $X$ we denote by $L(X)$ the Banach space of all bounded linear operators on $X$ equipped with
the standard norm. For basic notation for interpolation theory, we refer to \cite{BL} and \cite{BK}. We shall recall that a~mapping
$F\colon \vec{\mathcal{B}} \to \mathcal{B}$ from the category $\vec{\mathcal{B}}$ of all couples of Banach spaces into the category
$\mathcal{B}$ of all Banach spaces is said to be an \emph{interpolation functor} if, for any couple $\xo:=(X_0, X_1)$, the Banach
space $F(X_0, X_1)$ is  intermediate with respect to $\xo$ (i.e., $X_0 \cap X_1 \subset F(\xo) \subset X_0 + X_1$), and $T\colon
F(X_0, X_1) \to F(Y_0, Y_1)$ for all $T\colon (X_0, X_1) \to (Y_0, Y_1)$; here as usual  the notation $T\colon (X_0, X_1) \to
(Y_0, Y_1)$ means that $T\colon X_0 + X_1 \to Y_0 + Y_1$ is a~linear operator such that the restrictions of $T$ to the space $X_j$
is a~bounded operator from $X_j$ to $Y_j$, for both $j=0$ and $j=1$. An operator $T\colon (X_0, X_1) \to (Y_0, Y_1)$
between Banach couples is said to be invertible whenever the restriction $T|_{X_j}\colon X_j \to Y_j$ is invertible (i.e., $T$ is
an isomorphism of $X_j$ onto $Y_j$) for each $j\in \{0, 1\}$.

The complex method of interpolation plays an important role in applications in various areas of modern analysis. We point out
that in the study of spectral properties of interpolated operators between complex interpolation spaces the so called
uniqueness-of-resolvent property is of particular interest.

Let $\xo=(X_0, X_1)$ be a~complex Banach couple and $T\colon (X_0, X_1) \to (X_0, X_1)$ be an operator. If $0\leq \alpha <\beta \leq 1$
and $T_{\alpha}:=T|_{[\xo]_{\alpha}}$ and $T_{\beta}:=T|_{[\xo]_{\beta}}$ are invertible, then the inverses $T_{\alpha}^{-1}$ and
$T_{\beta}^{-1}$ do not coincide on $X_0 \cap X_1$ in general. Following Zafran \cite{Zafran}, an operator $T\colon \xo \to \xo$ is
said to have the uniqueness-of-resolvent (U.R.) property if the restrictions $(T_{\alpha} - \lambda I)^{-1}|_{X_0\cap X_1}$ and
$(T_{\beta} - \lambda I)^{-1}|_{X_0\cap X_1}$ coincide  for all $\alpha$, $\beta \in [0, 1]$ and $\lambda \notin \sigma(T_{\alpha})
\cup \sigma(T_{\beta})$.

Ransford \cite{Ransford} introduced a~weaker property; an operator $T\colon \xo \to \xo$ satisfies the local uniqueness-of-resolvent
(local U.R.) condition, if for all $\alpha \in (0, 1)$ and $\lambda \notin \sigma(T_{\alpha})$, there exists a~neighbourhood
$U\subset (0, 1)$ of $\alpha$ such that $(T_{\theta} - \lambda I)^{-1}$ exists and $(T_{\theta} - \lambda I)^{-1}|_{X_0\cap X_1}$
agrees with $(T_{\alpha} - \lambda I)^{-1}|_{X_0\cap X_1}$ for all $\theta \in U$. Albrecht and M\"uller proved in \cite{AM} that
this condition is always fulfilled. This follows immediately from the following result (see \cite[Theorem 4]{AM} which states:

\vspace{2 mm}

\noindent
\emph{If $(X_0, X_1)$ is a~complex Banach couple, $T\colon (X_0, X_1) \to (X_0, X_1)$ and $T_\alpha \colon [X_0, X_1]_\alpha
\to [X_0, X_1]_\alpha$ is invertible for some $\alpha \in (0, 1)$, then there exists a~neighbourhood $U\subset (0, 1)$
of $\alpha$ such that $T_{\theta}$ is invertible and $T_{\theta}^{-1}$ agrees with $T_{\alpha}^{-1}$ on
$X_0 \cap X_1$ for any  $\theta \in U$.}

\vspace{2 mm}

Our aim is to provide a~unified general approach to abstract compatibility theorems of stronger type than Albrecht--M\"uller
result for operators between Banach spaces generated by abstract interpolation methods. To do this we introduce a~new key notion
of a~\emph{stable family} $\{F_\theta\}_{\theta \in (0, 1)}$ of interpolation functors (for an exact definition
we refer to Section 4) and prove that certain class of interpolation methods introduced by Cwikel--Kalton--Milman--Rochberg
in \cite{CKMR} are stable. In particular, we prove that the Calder\'on complex family $\{[\,\cdot\,]_\theta\}_{\theta \in
(0, 1)}$ as well as the Lions--Peetre real family $\{(\,\cdot\,)_{\theta,q}\}_{\theta, q}$ for any $1\leq q\leq \infty$ of
interpolation functors are stable families.

The fundamental theorems on a~stable family $\{F_\theta\}_{\theta \in (0, 1)}$ of interpolation functors are
the following main results of this paper true for the restrictions $T_{\theta}:= T|_{F_\theta(X_0, X_0)}$
from $F_\theta(X_0, X_1)$ to $F_\theta(Y_0, Y_1)$ of any linear bounded operator $T\colon (X_0, X_1)\rightarrow (Y_0, Y_1)$
between \linebreak Banach couples:

\noindent
\emph{If $T_{\theta_{*}}\colon F_{\theta_{*}}(X_0, X_1) \to F_{\theta_{*}}(Y_0, Y_1)$ is invertible for some $\theta_{*}\in (0, 1)$,
then there exists a~neighbourhood $U\subset (0, 1)$ of $\theta_{*}$ such that $T_{\theta}$ is invertible, the inverse $T^{-1}_{\theta}$
agrees with $T^{-1}_{\theta_{*}}$ on $Y_0 \cap Y_1$ $($i.e., $T^{-1}_{\theta}(y) = T^{-1}_{\theta_{*}}(y)$ for all $y\in Y_0 \cap Y_1)$,
and the following estimate holds{\rm:}
\begin{align*}
\|T_{\theta}^{-1}\|_{F_\theta(Y_0, Y_1) \to F_\theta(X_0, X_1)} \leq
2\,\|T_{\theta_{*}}^{-1}\|_{F_{\theta_{*}}(Y_0, Y_1) \to F_{\theta_{*}}(X_0, X_1)} \quad\, \text{for all\, $\theta \in U$}.
\end{align*}
As a consequence, the set of all $\theta \in (0, 1)$ for which $T_\theta$ is invertible is an open subset of $(0, 1)$.}

We note that in addition we describe more precisely the mentioned above neighbourhood $U\subset (0, 1)$. Moreover, under some mild 
hypothesis on a~stable family $\{F_\theta\}$, which satisfies the reiteration condition, we prove a~subtle compatibility result which 
states{\rm:}

\noindent
\emph{If $I \subset (0, 1)$ is an open interval of invertibility of \,$T$ $($i.e., such that $T_\theta$ is invertible for all
$\theta\in I)$, then for any $\theta$, $\theta' \in I$ the inverse operators $T^{-1}_{\theta}$ and $T^{-1}_{\theta'}$
agree on $F_{\theta}(Y_0, Y_1) \cap F_{\theta'}(Y_0, Y_1)$.}

Among several motivations for studying compatibility problems are important applications to PDE's. It seems the roots for these problems
are in Calder\'on paper \cite{Ca2} in which it is proved that if $(\Omega, \Sigma, \mu)$ is a~measure space and $T\colon L^p(\mu) \to
L^p(\mu)$ is a~bounded operator for $1<p<\infty$, which is invertible for $p=2$, then $T$ is also invertible when $2-\varepsilon < p <
2+ \varepsilon$, for some small $\varepsilon>0$. In fact careful analysis of Calder\'on's proofs gives the compatibility of inverses,
i.e., there exists some small $\varepsilon >0$ such that for all $p, q \in (2-\varepsilon, 2+ \varepsilon)$, the inverse $T^{-1}$
considered on the space $L^p(\mu)$ is compatible with $T^{-1}$ considered on $L^q(\mu)$ when both operators are restricted to
$L^p(\mu) \cap L^q(\mu)$. It was shown in \cite{PV} very useful application for solvability of the Dirichlet problem with data
in $L^p(\partial{\Omega})$ for the biharmonic equation $\Delta u = 0$ in $\Omega$, $u = f$ and $\partial u/\partial n  = g$ on
$\partial \Omega$, in a bounded Lipschitz domain $\Omega \subset \mathbb{R}^n$.

It is worth pointing out that in the remarkable paper \cite{KMM} by Kalton--Mayaboroda--Mitrea there are shown
applications of compatibility results for the variants of the Dirichlet problem as well as the Neumann problem for the Laplacian
in $L_p(\partial \Omega)$-spaces in the case of unbounded domain $\Omega$ above the graph of a~real-valued Lipschitz function
defined in $\mathbb{R}^{n-1}$.

We conclude by noting that using the well known technics to the mentioned above type PDE's, our compatibility results can be applied 
to other methods than the complex. In particular, applying the real method, we would get variants of the Dirichlet problem as well as
the Neumann problem for the Laplacian in Lorentz $L_{p, q}(\partial\Omega)$-spaces.

Throughout the paper we shall require considerable notation. If $X$ and $Y$ are Banach spaces such that $X\subset Y$
and  the inclusion map $\text{id}\colon X \to Y$ is bounded, then we write $X \hookrightarrow Y$. For simplicity of notation,
we write $X\cong Y$ whenever $X=Y$, with \emph{equality} of norms.

\section{Notation and preliminary results}

We introduce the basic notations and definitions to be used throughout this work. We will use complex methods of
interpolation introduced by Calder\'on in his fundamental paper \cite{Ca1}.

Let $S:=\{z\in \mathbb{C};\, 0 <\text{Re} z<1\}$ be an open strip on the plane. For a~given $\theta \in (0, 1)$
and any couple $\xo = (X_0, X_1)$ we denote by $\mathcal{F}(\xo)$ the Banach space
of all bounded continuous functions $f\colon \bar{S}\to X_0 + X_1$ on the closure $\bar{S}$ that are analytic
on $S$, and $\mathbb{R} \ni t \mapsto f(j + it) \in X_j$ is a bounded continuous function, for each $j\in \{0,1\}$,
and endowed with the norm
\[
\|f\|_{\mathcal{F}(\xo)} = \max_{j=0, 1}\,\sup_{t\in \mathbb{R}} \|f(j+it)\|_{X_j}.
\]
The lower complex interpolation space is defined by $[\xo]_{\theta} := \{f(\theta);\, f\in \mathcal{F}(\xo)\}$ and
is endowed  with the quotient norm. This definition is slightly different from those in \cite{BL,Ca1},
however it gives the same interpolation spaces (see, e.g., \cite{Ca1}). We recall that in the original definition it is
required in addition that $f\in \mathcal{F}(\xo)$ satisfies
\[
\lim_{|t| \to \infty}\|f(j+i t)\|_{X_j}=0, \quad\, j\in \{0, 1\}.
\]
We also recall the basic constructions and results of \cite{CKMR} which we will use here, and we refer to
this paper for more details. Let {\bf Ban} be the class of all Banach spaces over the complex field. A mapping
$\mathcal{X}\colon {\bf Ban} \to {\bf Ban}$ is called a~pseudolattice, or a~{\bf pseudo}-$\mathbb{Z}$-{\bf
lattice}, if
\par (i) for every $B\in {\bf Ban}$ the space $\mathcal{X}(B)$ consists of $B$ valued sequences
$\{b_n\}=\{b_n\}_{n\in \mathbb{Z}}$ modelled on $\mathbb{Z}$;
\par (ii) whenever $A$ is a~closed subspace of $B$ it follows that $\mathcal{X}(A)$ is a~closed subspace
of $\mathcal{X}(B)$;
\par (iii) there exists a positive constant $C=C(\mathcal{X})$ such that, for all $A$, $B\in {\bf Ban}$
and all bounded linear operators $T\colon A\to B$ and every sequence $\{a_n\}\in \mathcal{X}(A)$, the
sequence $\{Ta_n\} \in \mathcal{X}(B)$ and satisfies the estimate
\[
\|\{Ta_n\}\|_{\mathcal{X}(B)} \leq C \|T\|_{A\to B}
\|\{a_n\}\|_{\mathcal{X}(A)};
\]
\par (iv)
\[
\|b_m\|_B \leq \|\{b_n\}\|_{\mathcal{X}(B)}
\]
for each $m\in \mathbb{Z}$, all $\{b_n\}\in \mathcal{X}(B)$ and all Banach spaces $B$.

For every Banach couple $\bo = (B_0, B_1)$ and  every Banach couple of pseudolattices $\bx=(\mathcal{X}_0, \mathcal{X}_1)$,
let $\mathcal{J}(\bx, \bo)$ be the Banach space of all $B_0\cap B_1$ valued sequences $\{b_n\}$ such that
$\{e^{jn}b_n\}\in \mathcal{X}(B_j)$ ($j=0,1$), equipped with the norm.
\[
\|\{b_n\}\|_{\mathcal{J}(\bx, \bo)} = \max\big\{\|\{b_n\}\|_{\mathcal{X}_0(B_0)}, \|\{e^n b_n\}\|_{\mathcal{X}_1(B_1)}\big\}.
\]

Following \cite{CKMR}, for every $s$ in the annulus $\mathbb{A}:= \{z\in \mathbb{C};\, 1<|z| < e\}$, we define the Banach
space $\bo_{\xo, s}$ to consist of all elements of the form $b=\sum_{n\in \mathbb{Z}} s^n b_n$ (convergence in $B_0 + B_1$
with $\{b_n\} \in \mathcal{J}(\bx, \bo)$, equipped with the norm
\[
\|b\|_{\bo_{\bx, s}} = \inf \bigg\{\|\{b_n\}\|_{\mathcal{J}(\bx, \bo)}; \,\, b= \sum_{n\in \mathbb{Z}} s^n b_n \bigg\}.
\]

It is easy to check that the map $\bo \mapsto \bo_{\bx, s}$ is an interpolation functor.

We will consider mainly couples $\bx = (\mathcal{X}_0, \mathcal{X}_1)$ of Banach pseudolattices,
which are translation invariant, i.e., such that any Banach space $B$ we have
\[
\big\|\{S^k(\{b_n\}_{n\in \mathbb{Z}}\big\}\big\|_{\mathcal{X}_j(B)} = \big\|\{b_n\}_{n\in \mathbb{Z}}\big\|_{\mathcal{X}_j(B)}
\]
for all $\{b_n\}_{n\in \mathbb{Z}}\in \mathcal{X}_j(B)$, each $k\in \mathbb{Z}$ and $j \in \{0,1\}$.
Here and in what follows $S$ denote the left-shift operator on two-sided (vector valued) sequences defined by
$S\{b_n\} = \{b_{n+1}\}$.

Following \cite{CKMR} $\bx = (\mathcal{X}_0, \mathcal{X}_1)$ is said to be a~\emph{rotation invariant}
Banach couple of pseudolattices whenever the \emph{rotation map} $\{b_n\}_{n\in \mathbb{Z}}
\mapsto \{e^{in\tau} b_n\}_{n\in \mathbb{Z}}$ is an isometry of $\mathcal{X}_j(B)$ onto itself for every
real $\tau$ and every Banach space $B$.

The following useful lemma is obvious, but we include a proof.

\begin{lemma}
\label{rotation}
Let $\bx= (\mathcal{X}_0, \mathcal{X}_1)$ be a~Banach couple of rotation invariant pseudolattices.
Then, for every Banach couple $\bo = (B_0, B_1)$ and all $s\in \mathbb{A}$, we have
\begin{itemize}
\item[{\rm(i)}] If $f\in \mathcal{F}_{\bx}(\bo)$, then $f(s) \in \bo_{\bx, |s|}$\,{\rm;}
\item[{\rm(ii)}] If $x\in \bo_{\bx, |s|}$, then there exists $f\in \mathcal{F}_{\bx}(\bo)$ such that
$f(s)=x$\,{\rm;}
\item[{\rm(iii)}] $\bo_{\bx, s} \cong \bo_{\bx, |s|}$.
\end{itemize}
\end{lemma}

\begin{proof} (i). Let  $f\in \mathcal{\mathcal{F}}_{\bx}(\bo)$. Then there exists $\{b_n\}_{n\in \mathbb{Z}}
\in \mathcal{J}(\bx, \bo)$ such that $f(z) = \sum_{n\in \mathbb{Z}} z^n b_n$ for all $z\in \mathbb{A}$
(convergence in $B_0 + B_1)$. Define $\widetilde{f}$ by $\widetilde{f}(z) = f(ze^{i \varphi})$ for all
$z\in \mathbb{A}$, where $\varphi:= \text{Arg}\,s$. Then
\[
\widetilde{f}(z) = \sum_{n\in \mathbb{Z}} z^{n} e^{in\varphi} b_n, \quad\, z\in \mathbb{A}.
\]
Our hypothesis yields $\|\{e^{-in\varphi} b_n\}\|_{\mathcal{J}(\bx, \bo)} =
\|\{b_n\}\|_{\mathcal{J}(\bx, \bo)}$ and so $\widetilde{f} \in \mathcal{F}_{\bx}(\bo)$. Since
$f(s) = \widetilde{f}(|s|)\in \bo_{\bx, |s|}$, $f(s) \in \bo_{\bx, |s|}$ and this proves (i).

(ii). Let $x\in \bo_{\bx, s}$. Then there exists $\widetilde{f} \in \mathcal{F}_{\bx, s}(\bo)$ with $b=\{b_n\}
\in \mathcal{J}(\bx, \bo)$ such that $\widetilde{f} \in \mathcal{F}_{\bx}(\bo)$ and $\widetilde{f}(|s|) = x$, where
\[
\widetilde{f}(z) = \sum_{n\in \mathbb{Z}} z^n b_n, \quad\, z\in \mathbb{A}.
\]
Define $f$ by $f(z) = \widetilde{f}(ze^{-i\varphi})$ for all $z\in \mathbb{A}$. Our hypothesis gives that
$f\in \mathcal{F}_{\bx}(\bo)$. Combining the above facts yields $f(s)= \widetilde{f}(|s|)=x$ and this proves (ii).

(iii). It is enough to observe that the proofs of (i) and (ii) yields
\[
\|f\|_{\mathcal{F}_{\bx}(\bo)}= \|\widetilde{f}\|_{\mathcal{F}_{\bx}(\bo)}, \quad\, f\in \mathcal{F}_{\bx}(\bo).
\]
\end{proof}

We note that the above lemma shows if $\bx= (\mathcal{X}_0, \mathcal{X}_1)$ is a~Banach couple of rotation invariant
pseudolattices, then for any $s = e^{\theta + i\varphi}$ with $\theta \in (0, 1)$ and $\varphi \in [0, 2\pi)$, we have
that $\bo_{\bx, s} \cong \bo_{\bx, e^{\theta}}$ for any Banach couple $\bo$.

We point out that concerning interpolation  methods the idea of \cite{CKMR} was to show that a~large family of interpolation 
methods have a~suitable complex analytic structure that could be used for methods that \emph{apriori} do not seem to have one. 
This essential fact is deeply used in our paper. Note that with the right choices of pseudolattice couples 
$(\mathcal{X}_0, \mathcal{X}_1)$, we recover the classical methods of interpolation (see \cite{CKMR} for more details).
In particular let $s= e^{\theta}$ with $0<\theta <1$. If $\mathcal{X}_0 = \mathcal{X}_1=\ell_p$ with $1\leq p\leq \infty$, 
the space $\bo_{\bx, s}$ coincides with the Lions--Peetre real $J$-method space $\bo_{\theta, p;J}$ (see, e.g., \cite[p.~41]{LP} 
where this space is denoted by $s(p, \theta, B_{0}; p, \theta-1, B_1)$.

It is well known that  $(B_0, B_1)_{\theta, p; J}= (B_0, B_1)_{\theta, p}$ up to equivalence of norms (see \cite[Chap.~3]{BL}), 
where $(B_0, B_1)_{\theta, p}$ is the $K$-method space endowed with the norm
\[
\|b\|_{\theta, p}: = \bigg(\int_0^{\infty} \big(t^{-\theta} K(t, b; \bo)\big)^{p}\frac{dt}{t}\bigg)^{1/p}, \quad\, 1\leq p < \infty.
\]
For $\theta \in [0, 1]$ and $p=\infty$ the real interpolation space $\bo_{\theta, \infty}$ is defined to be a~space of all 
$b\in B_0 + B_1$ endowed with the norm
\[
\|b\|_{\theta, \infty}: = \sup_{t>0} t^{-\theta} K(t, b; \bo).
\]
Here as usual for any Banach couple $\xo = (X_0, X_1)$ the Peetre $K$-functional is defined  by
\[
K(t, x; \xo) = K(t, x; X_0, X_1): = \inf\{\|x_0\|_{X_0} + \|x_1\|_{X_1}; \, x_0 + x_1 = x\}, \quad\, t>0.
\]
Let $X$ be a Banach space intermediate with respect to a~Banach couple $\xo=(X_0, X_1)$. The
\emph{Gagliardo completion} or \emph{relative completion} of $X$  with respect to $\xo$ is the Banach space
$X^{\rm c}$ of all limits in $X_0 + X_1$ of sequences that are bounded in $X$ and endowed with the norm
$\|x\|_{X^{\rm c}} =\inf \{\sup_{k\geq 1} \|x_k\|_{X}\}$, where the infimum is taken over all bounded
sequences $\{x_k\}$ in $X$ whose limit in $X_0+X_1$ equals $x$.

We will use without any references the well-known fact (see \cite[Lemma 2.2.30]{BK}) that for any Banach
couple $(X_0, X_1)$ we have
\[
(X_0)^{c} \cong (X_0, X_1)_{0, \infty}\,, \quad\,  (X_1)^{c} \cong (X_0, X_1)_{1, \infty}.
\]

If $\bx= (FC, FC)$, then $\bo_{\bx, s}$ coincides, to within equivalence of norms, with the Cader\'on complex
method space $[\bo]_{\theta} = [B_0, B_1]_{\theta}$ (see \cite{Cwikel}). If $\bx = (UC, UC)$, then $\bo_{\bx, s}$
is the $\pm$ method space $\langle \bo \rangle_{\theta} \cong \langle B_0, B_1 \rangle_{\theta}$
(see \cite[p.~176]{Peetre}). If we replace $UC$ by $WUC$, we obtain the Gustavsson--Peetre variant of
$\langle B_0, B_1 \rangle_{\theta}$ which is denoted  by $\langle \bo; \theta \rangle$ (see \cite[p.~45]{GP}, \cite{Janson}).

\section{The uniqueness of inverses on intersection of a~couple}

Throughout the paper, for an operator $T\colon \xo \to \yo$ between Banach couples and every $\omega \in \mathbb{A}$,
we often denote by $T_{\omega}$ the restriction  $T|_{\xo_{\bx, \omega}} \colon \xo_{\bx, \omega} \to \yo_{\bx, \omega}$.
For simplicity of notation, we write $T_{\theta}$ instead of $T_{e^{\theta}}$ for any $\theta \in (0, 1)$.

We state the main results of this section for operators between spaces generated by interpolation constructions described
in the previous section.

\begin{theorem}
\label{inverse}
Let $\bx = (\mathcal{X}_0, \mathcal{X}_1)$ be a~Banach couple of translation invariant pseudolattices and let $T\colon \xo \to \yo$
be an operator between complex Banach couples. Assume that $T\colon \xo_{\bx, s} \to \yo_{\bx, s}$ is invertible
for some $s \in \mathbb{A}$. Then $T_{\omega}\colon \xo_{\bx, \omega} \to \yo_{\bx, \omega}$  is invertible for
all $\omega$ in an open neighbourhood $W = \{\omega\in \mathbb{A};\, |\omega-s| < r\}$ of $s$ in $\mathbb{A}$ with
\[
r = \big[2\delta(s)\big(1 + \|T\|_{\xo \to \yo}\|T^{-1}\|_{\yo_{\bx, s} \to \xo_{\bx, s}}\big)\big]^{-1},
\]
\noindent
where $\delta(s) = \max\big\{(|s|-1)^{-1},\,(e - |s|)^{-1}\big\}$. Moreover the following upper estimate
for the norm of $T_\omega$ holds,
\[
\big\|T_{\omega}^{-1}\big\|_{\yo_{\bx, \omega} \to \xo_{\bx, \omega}}
\leq 2 \big\|T_s^{-1}\big\|_{\yo_{\bx, s} \to \xo_{\bx, s}}, \quad\, \omega \in W.
\]
\end{theorem}

In the case when $\bx = (\mathcal{X}_0, \mathcal{X}_1)$ is a~couple of translation and rotation invariant pseudolattices
we obtain the following variant of Albrechr--Miller result.

\begin{theorem}
\label{uinverse}
Let $\bx = (\mathcal{X}_0, \mathcal{X}_1)$ be a~couple of translation and rotation invariant pseudolattices
and let $T\colon \xo \to \yo$
be an operator between complex Banach couples. Assume that $T_{\theta_{*}} \colon \xo_{\bx, e^{\theta_{*}}} \to
\yo_{\bx, e^{\theta_{*}}}$ is invertible for some $\theta_{*} \in (0, 1)$. Then $T_{\theta}\colon \xo_{\bx, e^\theta}
\to \yo_{\bx, e^\theta}$ is invertible  for all $\theta$ in an open neighbourhood
$I= \{\theta\in (0, 1);\, |\theta- \theta_{*}| < \varepsilon\}$ of $\theta_{*}$ with
\[
\varepsilon = \big[2e\eta(\theta_{*})\big(1 + \|T\|_{\xo \to \yo}\|T^{-1}\|_{\yo_{\bx, e^{\theta_{*}}}} \to
\xo_{\bx, e^{\theta_{*}}}\big)\big]^{-1},
\]
\noindent
where $\eta(\theta_{*}) = \max\big\{(e^{\theta_{*}} - 1)^{-1},\,(e - e^{\theta_{*}})^{-1}\big\}$. Moreover,
$T_{\theta}^{-1}$ agrees with $T_{\theta_{*}}$ on $Y_0 \cap Y_1$ and
\begin{align*}
\big\|T_{\theta}^{-1}\big\|_{\yo_{\bx, e^\theta} \to \xo_{\bx, e^\theta}}
\leq 2 \big\|T_{\theta_{*}}^{-1}\big\|_{\yo_{\bx, e^{\theta_{*}}} \to \xo_{\bx, e^{\theta_{*}}}} \quad\,
\text{for any \,$\theta \in I$}.
\end{align*}
\end{theorem}

To prove this theorem we will need some preliminary results. We start our investigation with the following
a~more precise cancellation principle (cf. \cite{CKMR}. We note that careful analysis
of the proof Lemma 3.1 in \cite{CKMR} gives a~key Lipschitz estimate with a~constant depending on parameter
$s\in \mathbb{A}$, but not on the couple of translation invariant pseudolattices. Since this estimate is
essential in our study, we include a~proof for readers' convenience.

\begin{lemma}
\label{cancellation}
Let $\bx$ be a~couple of translation invariant pseudolattices and let $\bo=(B_0, B_1)$ be a~Banach couple. Let the
sequence $\{f_n\}_{n\in \mathbb{Z}}$ be an element of $\mathcal{J}(\bx, \bo)$ and let $f\colon \mathbb{A} \to B_0 + B_1$
be the analytic function defined by $f(z) = \sum_{n\in \mathbb{Z}} z^n f_n$. Suppose that $f(s)=0$ for some $s\in \mathbb{A}$
and let $g\colon \mathbb{A} \to B_0 + B_1$ be given by $g(s)= f'(s)$ and $g(z)= f(z)/(z-s)$ for all
$z\in \mathbb{A}\setminus \{s\}$. Then, $g\in \mathcal{F}_{\bx}(\bo)$ and  the Laurent expansion of $g$ in $\mathbb{A}$,
$g(z)= \sum_{n\in \mathbb{Z}} z^n g_n$ for all $\mathbb{A}$, satisfies $\{g_n\}_{n\in \mathbb{Z}}
\in \mathcal{J}(\bx, \bo)$ and
\[
\|\{g_n\}\|_{\mathcal{J}(\bx, \bo)} \leq \delta(s) \|\{f_n\}\|_{\mathcal{J}(\bx, \bo)},
\]
where $\delta(s) = \max\big\{(|s|-1)^{-1},\,(e - |s|)^{-1}\big\}$.
\end{lemma}

\begin{proof}
Let $f \in \mathcal{F}_{\bx}(\bo)$ be such that $f(s)=0$. We define
\[
g(z) =
\begin{cases}
\frac{f(z)}{z-s}\,, &\text{if } z \neq s \\
f'(z)\,, &\text{if } z= s\,.
\end{cases}
\]
Clearly $g\colon \mathbb{A} \to B_0 + B_1$ is analytic. Let $f(z) = \sum_{n\in \mathbb{Z}} f_n z^n$
for all $z\in \mathbb{A}$, where $\{f_n\}\in \mathcal{J}(\bx, \bo)$. We claim that the Laurent expansion of $g$ in
$\mathbb{A}$ satisfy the required properties. Because of the uniqueness of the Laurent expansion, it is enough to
show that $g(z)= \sum_{n\in \mathbb{Z}} z^n g_n$ for all $|s| <|z| <e$, and moreover that
$\{g_n\} \in \mathcal{J}(\bx, \bo)$ satisfies the desired estimate.

Fix $z\in \mathbb{A}$ such that $|z|>|s|$. Combining the absolute convergence of series, we have
\begin{align*}
g(z )& = \frac{1}{z\big(1 - \frac{s}{z}\big)} \sum_{n \in \mathbb{Z}} f_n z^n
=  \sum_{k\geq 0} \frac{1}{z} \Big(\frac{s}{z}\Big)^k\, \sum_{n\in \mathbb{Z}} f_n z^n \\
& = \sum_{n\in \mathbb{Z}} \sum_{k\geq 0} z^{n-k-1} s^k f_n
= \sum_{m\in \mathbb{Z}} \Big(\sum_{k\geq 0} s^k f_{m+k+1}\Big) z^m.
\end{align*}
Since $f(s)= 0$, we get that $\sum_{n\geq k} s^n f_n = - \sum_{n<k} s^n f_n$ for each $k\in \mathbb{Z}$.
This implies that the sequence $\{g_n\}$ defined by
\[
g_{n} := \sum_{k\geq 0} s^k f_{n+k+1} = - \sum_{k<0} s^k f_{n+k+1}, \quad\, n\in \mathbb{Z}
\]
satisfies
\begin{align*}
\|\{g_n\}\|_{\mathcal{X}_0(B_0)} & = \Big\|\Big\{\sum_{k<0} s^{k} f_{n+k+1}\Big\}_n\Big\|_{\mathcal{X}_0(B_0)}
= \Big\|\sum_{k<0} s^k S^{k+1}(\{f_n\})\Big\|_{\mathcal{X}_0(B_0)} \\
& \leq \sum_{k<0} |s|^k \big\|S^{k+1}(\{f_n\})\big\|_{\mathcal{X}_0(B_0)}  \leq  \frac{1}{|s|-1}\,\|\{f_n\}\|_{\mathcal{X}_0(B_0)}
\end{align*}
and
\begin{align*}
\|\{e^{n}g_n\}\|_{\mathcal{X}_1(B_1)} & = e^{-1} \Big\|\Big\{\sum_{k\geq 0}
\frac{s^k}{e^k}\,e^{n+k+1}\Big\}_n f_n\Big\|_{\mathcal{X}_1(B_1)}  \\
& \leq e^{-1} \sum_{k\geq 0} \frac{|s|^k}{e^k}\,\big\|S^{k+1}(\{e^n f_n\})\big\|_{\mathcal{X}_1(B_1)} \\
& \leq e^{-1} \Big(\sum_{k\geq 0} \frac{|s|^k}{e^k} \Big)\,\|\{e^n f_n\}\|_{\mathcal{X}_1(B_1)} \\
& = \frac{1}{e - |s|}\,\|\{e^n f_n\}\|_{\mathcal{X}_1(B_1)}.
\end{align*}
The above estimates proves the claim and this completes the proof.
\end{proof}

Now, we introduce special maps and spaces which will play an essential role. Let $\bx = (\mathcal{X}_0, \mathcal{X}_1)$
be a~couple of pseudolattices and $\bo=(B_0, B_1)$ a~Banach couple. For our purposes it will be convenient to express
a~natural correspondence between elements in the space $\mathcal{J}(\bx, \bo)$ and certain analytic functions defined
on $\mathbb{A}$ with values on $B_0 + B_1$. To see this we define the space $\mathcal{F}_{\bx}(\bo)$ to consist of all
vector valued analytic functions $f_{b}\colon \mathbb{A} \to B_0 + B_1$ which has the Laurent series expansion
given by
\[
f_{b}(z)= \sum_{n\in \mathbb{Z}} z^n b_n, \quad\, z\in \mathbb{A}
\]
for some $b=\{b_n\} \in \mathcal{J}(\bx, \bo)$.

Since $\mathcal{J}(\bx, \bo)$ is a~Banach space, the uniqueness theorem for analytic functions implies that
$\mathcal{F}_{\bx}(\bo)$ is a~Banach space isometrically isomorphic to $\mathcal{J}(\bx, \bo)$ whenever
$\mathcal{F}_{\bx}(\bo)$ is equipped with the norm
\[
\|f_b\|_{\mathcal{F}_{\bx}(\bo)} =
\|\{b_n\}\|_{\mathcal{J}(\bx,\bo)}.
\]
The kernel of the continuous map $\delta_s \colon \mathcal{F}_{\bx}(\bo) \to B_0 + B_1$, given by
$\delta_s(f)= f(s)$, for all $f\in \mathcal{F}_{\bx}(\bo)$ is denoted by $N_s(\bo)$, i.e.,
\[
N_s(\bo)= \big\{f\in \mathcal{F}_{\bx}(\bo); \,\, f(s)=0\big\}.
\]
Clearly, the map $\widehat{\delta}_s\colon \mathcal{F}_{\bx}(\bo)/N_s(\bo) \to \bo_{\bx, s}$ defined by
\[
\widehat{\delta_s}(f + N_s(\bo)) = \delta_s(f), \quad\, f + N_s(\bo) \in \mathcal{F}_{\bx}(\bo)/N_s(\bo)
\]
is an isometrical isomorphism of $\mathcal{F}_{\bx}(\bo)/N_s(\bo)$ onto $\bo_{\bx, s}$.

In what follows we will apply a result from \cite{KM}. For the reader's convenience, we state this result.
To do this we need to recall some fundamental definitions from the theory of distances between closed subspaces
of Banach spaces.

Let $U$ be a Banach space. For two given closed subspaces $U_0$, $U_1$ of $U$ we let
\[
\text{dist}(U_0, U_1) := \sup_{\|u\|_U = 1} |\text{dist}(u, U_0) - \text{dist}(u, U_1)|,
\]
where for any $u \in U$,
\[
\text{dist}(u, U_j) = \inf_{u_j \in U_j} \|u - u_j\|_{U}, \quad\, j\in \{0,1\}.
\]

Let $U$, $V$ be Banach spaces and let $U_0$, $U_1$ and $V_0$, $V_1$ be closed subspaces of $U$ and $V$, respectively.
Let $H$ be a linear bounded operator from $U$ to $V$ which maps $U_j$ to $V_j$ for $j\in \{0,1\}$. Since $H(u + u_j)
= H(u) + H(u_j) \in H(u) + V_j$ for all $u_j \in U_j$, we can define quotient operators $H_{j}\colon U/U_j \to V/V_{j}$
for each $j\in \{0,1\}$ by
\[
H_j(u + U_j) :=  H(u) + V_j, \quad\, u + U_j \in U/U_{j}.
\]

In what follows the following theorem is the crucial tool. The proof is a straightforward minor modification of the
proof of Theorem 9 in \cite{KM}.

\begin{theorem}
\label{quotient}
Suppose that $H\colon U\to V$ maps $U_j$ to $V_j$ for each $j\in \{0, 1\}$, and the quotient operator $H_0\colon U/U_0 \to V/V_0$
is invertible. If
\[
\max\{{\rm{dist}}(U_0, U_1), {\rm{dist}}(V_0, V_1)\} < \frac{1}{2\big(1 +  \|H\|_{U\to V} \|H_{0}^{-1}\|_{V/V_0 \to U/U_0}\big)},
\]
then the quotient operator $H_1\colon U/U_1 \to V/V_1$ is invertible. Moreover the upper estimate for the norm of $H_1$ is given by
\[
\|H_{1}^{-1}\|_{V/V_1 \to U/U_1} \leq 2 \|H_{0}^{-1}\|_{V/V_0 \to U/U_0}.
\]
\end{theorem}

\vspace{2 mm}

Let $\bx$ be a~Banach couple of pseudolattices, $\bo$ a~Banach couple, and let $``\text{dist}$'' be a~distance defined
on closed subspaces of the space $\mathcal{F}_{\bx}(\bo)$, and let $s$, $\omega \in \mathbb{A}$. Then we define
\[
\rho\big(\bo_{\bx, s}, \bo_{\bx, \omega}\big) := \text{dist}\big(N_s(\bo),\, N_{\omega}(\bo)\big).
\]

The following variant of a result from \cite{KM} is relevant to our purposes.

\begin{theorem}
\label{dist}
Let $\bo$ be a~complex Banach couple. Then, for all $s\in \mathbb{A}$,
\begin{align*}
{\rm{dist}}\big((\cdot)_{\bx, s}, (\cdot)_{\bx, \omega}\big) :=
\sup_{\bo \vec{\mathcal{B}}}\rho\big(\bo_{\bx, s}, \bo_{\bx, \omega}\big) \leq \delta(s)\,|\omega - s|,
\quad\, \omega \in \mathbb{A},
\end{align*}
where $\delta(s) = \max\big\{(|s|-1)^{-1},\,(e - |s|)^{-1}\big\}$.
\end{theorem}

\begin{proof} We have
\begin{align*}
\rho\big(\bo_{\bx, s}, \bo_{\bx, \omega}\big) & = \sup_{\|f\|_{\mathcal{F}_{\bx}(\bo)}
\leq 1} \big|\rho\big(f, N_s(\bo)\big) - \rho\big(f, N_{\omega}(\bo)\big)\big| \\
& = \sup_{\|f\|_{\mathcal{F}_{\bx}(\bo)} \leq 1} \Big |\big\|f+ N_s(\bo)\big\|_{\mathcal{F}_{\bx}(\bo)/N_s(\bo)}
-\big\|f+ N_{\omega}(\bo)\big\|_{\mathcal{F}_{\bx}(\bo)/N_{\omega}(\bo)}\Big| \\
& =\sup_{\|f\|_{\mathcal{F}_{\bx}(\bo)} \leq 1} \Big|\|f(s)\|_{\bo_{\bx, s}}  - \|f(\omega)\|_{\bo_{\bx, \omega}}\Big|.
\end{align*}
Let $f \in \mathcal{F}_{\bx}(\bo)$ be such that $\|f\|_{\mathcal{F}_{\bx}(\bo)} \leq 1$, and let $x=f(s)$.
Given $\varepsilon >0$ select $f_x \in \mathcal{F}_{\bx}(\bo)$ such that
\[
f_x(s) = x, \quad\, \|f_x\|_{\mathcal{F}_{\bx}(\bo)} \leq \|x\|_{\bo_{\bx, s}} + \varepsilon.
\]
In particular we have $\|f_x\|_{\mathcal{F}_{\bx}(\bo)} \leq 1 + \varepsilon$. Since $f(s) - f_x(s)=0$,
it follows from  Lemma \ref{cancellation} that the function $h$ defined by $h(s) = f'(s) - f_x'(s)$ and
\[
h(z) = \frac{f(z)- f_x(z)}{z-s}, \quad\, z\in
\mathbb{A}\setminus\{s\}
\]
is in $\mathcal{F}_{\bx}(\bo)$, and
\[
\|h\|_{\mathcal{F}_{\bx}(\bo)} \leq \delta(s)
\]
for some positive constant $\delta(s) \leq \max\big\{(|s|-1)^{-1},\,(e - |s|)^{-1}\big\}$.

Now observe that
\[
f(\omega) - f_x(\omega) = (\omega-s)h(\omega), \quad\, \omega\in \mathbb{A}
\]
and so
\[
\|f(\omega)- f_x(\omega)\|_{\bo_{\bx, \omega}} \leq |\omega-s| \|h(\omega)\|_{\bo_{\bx,
\omega}} \leq \delta(s)\,|\omega - s|.
\]
Combining the above facts with the triangle inequality yields that, for all $\omega\in \mathbb{A}$,
\begin{align*}
\|f(\omega)\|_{\bo_{\bx, \omega}} & \leq \|f_x(\omega)\|_{\bo_{\bx, \omega}} + \delta(s)\,|\omega -s| \\
& \leq \|f_x\|_{\mathcal{F}_{\bx}(\bo)} + \delta(s)\,|\omega - s| \\
& \leq \|x\|_{\bo_{\bx, s}} + \delta(s) \,|\omega - s| + \varepsilon \\
& \leq \|f(s)\|_{\bo_{\bx, s}} + \delta(s) \,|\omega - s| + \varepsilon.
\end{align*}
Since $\varepsilon$ is arbitrary, we get
\[
\Big|\|f(\omega)\|_{\bo_{\bx, \omega}} -  \|f(s)\|_{\bo_{\bx, s}}\Big|
\leq \delta(s) \,|\omega - s|, \quad\, \omega\in \mathbb{A},
\]
and this completes the proof.
\end{proof}

\noindent
We are ready for the  proof Theorem \ref{inverse}.

\begin{proof}[Proof of Theorem {\rm{\ref{inverse}}}]
For $\omega \in \mathbb{A}$ define the operator $\widetilde{T}_{\omega}\colon
\mathcal{F}_{\bx}(\xo)/N_{\omega}(\xo) \to \mathcal{F}_{\bx}(\yo)/N_{\omega}(\yo)$ by
\[
\widetilde{T}_w(f+ N_{\omega}(\xo)) = \widetilde{T}f + N_{\omega}(\yo), \quad\, f+ N_{\omega}(\xo)\in
\mathcal{F}_{\bx}(\xo)/N_{\omega}(\xo),
\]
where $\widetilde{T}\colon \mathcal{F}_{\bx}(\xo) \to \mathcal{F}_{\bx}(\yo)$ is the operator given by
\[
\widetilde{T}f(z)= T(f(z)), \quad\, f\in \mathcal{F}_{\bx}(\xo), \quad\, z\in \mathbb{A}.
\]
We note that $\|\widetilde{T}\|_{\mathcal{F}_{\bx}(\xo) \to \mathcal{F}_{\bx}(\yo)}
\leq \|T\|_{\xo \to \yo}=\max_{j=0, 1}\|T\|_{X_j \to Y_j}$ and
\[
\tag {$*$}\, \|\widetilde{T}_{\omega}\|_{\mathcal{F}_{\bx}(\xo)/N_{\omega}(\xo) \to \mathcal{F}_{\bx}(\yo)/N_{\omega}(\yo)}
= \|T\|_{\xo_{\bx, \omega} \to \yo_{\bx, \omega}}.
\]

Now we fix $s\in \mathbb{A}$. Then, from Theorem \ref{dist}, we conclude that for
$\delta(s) = \max\big\{(|s|-1)^{-1},\,(e - |s|)^{-1}\big\}$ we have
\[
{\rm{dist}}\big((\cdot)_{\bx, s}, (\cdot)_{\bx, \omega}\big) \leq \delta(s) |\omega- s|, \quad\,\omega\in \mathbb{A}.
\]
Let $W := \{\omega\in \mathbb{A};\, |\omega-s| < r\}$ be an open neighbourhood of $s$ in $\mathbb{A}$ with
\[
r = \big(2\delta(s) + 2\delta(s)\,\|T^{-1}\|_{\yo_{\bx, s} \to \xo_{\bx, s}} \|T\|_{\xo \to \yo}\big)^{-1}.
\]
Then, we have
\[
{\rm{dist}}\big((\cdot)_{\bx, s}, (\cdot)_{\bx, \omega}\big) < \frac{1}
{2\big(1 +  \|T\|_{\xo \to \yo} \|T^{-1}\|_{\yo_{\bx, s} \to \xo_{\bx, s}}\big)}, \quad\, \omega\in W.
\]
Combining the above with Theorem \ref{quotient} applied to the Banach spaces $U= \mathcal{F}_{\bx}(\xo)$,
$V= \mathcal{F}_{\bx}(\yo)$, closed subspaces $U_{0} = N_s(\xo)$, $U_{1} = N_{\omega}(\xo)
\subset U$ and $V_{0} = N_s(\yo)$, $V_{1} = N_{\omega}(\yo) \subset V$  with $\omega\in W$, and operators
$H = \widetilde{T}$, $H_{0} = \widetilde{T}_s$, $H_{1} = \widetilde{T}_{\omega}$, we obtain the desired
statement for shown above open neighborhood $W\subset \mathbb{A}$ of $s$.

To get the required estimate for the norm of $T_{\omega}^{-1}$ for all $\omega \in W$, we first observe that following
the above notation, it follows from equality $(*)$ that $\|H\|_{U\to V} = \|T\|_{\xo \to \yo}$ and
\[
\|H_{1}^{-1}\|_{V/V_1 \to U/U_1}= \|T_{\omega}^{-1} \|_{\yo_{\bx, \omega} \to \xo_{\bx, \omega}}, \quad\,
\|H_{0}^{-1}\|_{V/V_0 \to U/U_0}= \|T_{\omega}^{-1} \|_{\yo_{\bx, s} \to \xo_{\bx, s}}.
\]
By Theorem \ref{dist}, for all $\omega \in W$, we have
\[
\max\{{\rm{dist}}(U_0, U_1),  {\rm{dist}}(V_0, V_1)\} \leq
{\rm{dist}}\big((\cdot)_{\bx, s}, (\cdot)_{\bx, \omega}\big) < \frac{1}{2(1 +  \|T\|_{\xo \to \yo}
\|T^{-1}\|_{\yo_{\bx, s} \to \xo_{\bx, s}})}.
\]
To finish, we apply Theorem \ref{quotient} to get the required estimate.
\end{proof}

\vspace{2 mm}

We isolate the following simple proposition for further reference.

\begin{proposition}
\label{RN} Let $\bx = (\mathcal{X}_0, \mathcal{X}_1)$ be a~couple of pseudolattices and let $\yo$ be a~Banach couple.
Then, for every $\omega \in \mathbb{A}$, the operator $V_{\omega}\colon \mathcal{F}_{\bx}(\yo) \to \mathcal{F}_{\bx}(\yo)$
defined by
\[
(V_{\omega}f)(z) = (\omega - z)f(z), \quad\, f\in \mathcal{F}_{\bx}(\yo), \quad\, z\in \mathbb{A}
\]
is injective and it has closed range with $R(V_{\omega}) = N_{\omega}(\yo)$.
\end{proposition}

\begin{proof} We first remark that our hypothesis on $\bx$ yields that a~function  $\mathbb{A} \ni z\mapsto z f(z)
\in \mathcal{F}_{\bx}(\yo)$ for any $f\in \mathcal{F}_{\bx}(\yo)$. Thus the domain $D(V_{\omega}) = \mathcal{F}_{\bx}(\yo)$.
Clearly, $V_{\omega}f= 0$ for $f\in \mathcal{F}_{\bx}(\yo)$ implies that $f(z)= 0$ for all $z\in \mathbb{A}\setminus \{\omega\}$
and whence $f= 0$ by continuity of $f$.

It is obvious that the range satisfies
\[
R(V_{\omega}) \subset \{g\in \mathcal{F}_{\bx}(\yo);\, g(\omega) = 0\} = N_{\omega}(\yo).
\]
To show the reverse inclusion let $g\in \mathcal{F}_{\bx}(\yo)$ with $g(\omega)=0$. It follows from Lemma \ref{cancellation}
that there exists a function $f\in \mathcal{F}_{\bx}(\yo)$ such that
\[
g(z)= (\omega- z)f(z), \quad\, z\in \mathbb{A}.
\]
Thus, we get that $g= V_{\omega}f$ and so the desired equality $R(V_{\omega}) = N_{\omega}(\yo)$ holds. Since
$N_{\omega}(\yo)$ is a~closed subspace in $\mathcal{F}_{\bx}(\yo)$, the proof is complete.
\end{proof}

We prove a~lemma which will play a~key role in the proof of the main result, Theorem \ref{uinverse}. In the proof
we will use some methods from \cite[Theorem 4]{AM}. We recall that if $S\colon X \to Y$ is a~bounded linear operator
between Banach spaces, then, the so called \emph{lower bound} of $S$ is defined by
\[
\gamma(S)= \inf \{\|Sx\|_{Y};\, x\in X, \, \|x\|_X = 1\}.
\]
It is obvious that $\gamma(S) > 0$ if, and only if, $S$ is injective and the range $R(S)$ of $S$ is a~closed
subspace in $Y$.

\begin{lemma}
\label{onto}
Let $\bx$ be a~couple of pseudolattices and let $\xo= (X_0, X_1)$, $\yo = (Y_0, Y_1)$ be complex Banach
couples, $T\colon \xo \to \yo$ and $s \in \mathbb{A}$. Assume that $T\colon \xo_{\bx, s} \to \yo_{\bx, s}$
is invertible. Then, there exists an open neighborhood $U\subset \mathbb{A}$ of $s$ such that, for all
$k\in \mathcal{F}_{\bx}(\yo)$, there exist analytic functions $g\colon U \to \mathcal{F}_{\bx}(\xo)$
and $h\colon U \to \mathcal{F}_{\bx}(\yo)$ such that, for all $\omega \in U$,
\[
T(g(\omega)(z)) + (\omega - z) h(\omega)(z)=k(z), \quad\, z\in \mathbb{A}.
\]
\end{lemma}

\begin{proof} From Proposition \ref{RN}, it follows that the injective operator $V_{\omega}\colon
\mathcal{F}_{\bx}(\yo) \to \mathcal{F}_{\bx}(\yo)$ given for every $\omega\in \mathbb{A}$ by
\[
V_{\omega}f(z)= (\omega-z)f(z), \quad\, f\in \mathcal{F}_{\bx}(\yo), \quad\, z\in \mathbb{A}.
\]
has the closed range $R(V_{\omega})= N_{\omega}(\yo)$. Thus, the lower bound $\gamma(V_{\omega})>0$ for
all $\omega\in \mathbb{A}$. Since
\[
V_{\omega} = (\omega-s)\,{\rm{id}}_{\mathcal{F}_{\bx}(\yo)} + V_{s}, \quad\, \omega\in \mathbb{A},
\]
$\mathbb{A} \ni \omega \mapsto V_{\omega} \in L\big(\mathcal{F}_{\bx}(\yo), \mathcal{F}_{\bx}(\yo)\big)$
is an analytic function.

We shall adopt notations from Theorem \ref{inverse}. Thus we will consider operators
$\widetilde{T}\colon \mathcal{F}_{\bx}(\xo) \to \mathcal{F}_{\bx}(\yo)$ and
$\widetilde{T}_\omega \colon \mathcal{F}_{\bx}(\xo)/N_{\omega}(\xo)
\to \mathcal{F}_{\bx}(\yo)/N_{\omega}(\yo)$, where $\omega \in \mathbb{A}$.

We note that $\|\widetilde{T}\|_{\mathcal{F}_{\bx}(\xo) \to
\mathcal{F}_{\bx}(\yo)} \leq \max_{j=0, 1}\|T\|_{X_j \to Y_j}$ and
\[
\|\widetilde{T}_{\omega}\|_{\mathcal{F}_{\bx}(\xo)/N_{\omega}(\xo) \to \mathcal{F}_{\bx}(\yo)/N_{\omega}(\xo)}
= \|T\|_{\xo_{\bx, \omega} \to \yo_{\bx, \omega}}.
\]
Let $c_1$ and $c$ be positive constants such that
\[
c_1> \|T_{s}^{-1}\| \quad\, \text{and} \quad\, c> \big(1 + c_1\|\widetilde{T}\|\big)\gamma(V_{s})^{-1},
\]
where, for simplicity of notation, we let $\|T_{s}^{-1}\| = \|T^{-1}\|_{\yo_{\bx, s} \to \xo_{\bx,s}}$,
$\|\widetilde{T}\| =\|\widetilde{T}\|_{\mathcal{F}_{\bx}(\xo) \to \mathcal{F}_{\bx}(\yo)}$.

It follows from Theorem \ref{inverse} that there exists an open neighbourhood $W\subset \mathbb{A}$ of $s$
such that $T_{\omega}\colon \xo_{\bx, \omega} \to \xo_{\bx, \omega}$ is invertible for all $\omega\in W$.

We claim that an open neighborhood $U\subset \mathbb{A}$ of $s$ given by
\[
U:= \big\{\omega\in \mathbb{A}; \, |\omega-s| < c^{-1}\big\} \cap W
\]
satisfies the required statements, i.e., there exist analytic function $g\colon U \to \mathcal{F}_{\bx}(\xo)$
and $h\colon U \to \mathcal{F}_{\bx}(\yo)$ such that
\[
\widetilde{T}g(\omega) + V_{\omega}h(\omega) = k, \quad\, \omega \in U.
\]
To see this fix $k\in \mathcal{F}_{\bx}(\yo)$ and observe that, if  $g(\omega) =
\sum_{n=0}^{\infty} g_n (\omega - s)^n$ and $h(\omega)= \sum_{n=0}^{\infty}h_n (\omega-s)^n$ are the Taylor
expansions of $g$ and $h$ about $s$, then solution of the required equation
\[
\widetilde{T}g(\omega) + V_{\omega}h(\omega) = k, \quad\, \omega\in U
\]
with $g$ and $h$ in the form given above reduces to solution of the following recurrence equations generated
by the sequences $\{g_n\}\subset \mathcal{F}_{\bx}(\xo)$ and $\{h_n\}\subset \mathcal{F}_{\bx}(\yo)$ of
Taylor's coefficients of $g$ and $h$, respectively
\[
\widetilde{T}g_0 + V_{s}h_0 = k,
\]
\[
\widetilde{T}g_n + V_{s}h_n = - h_{n-1}, \quad\, n\in \mathbb{N}
\]
such that the both series $\sum_{n=0}^{\infty} g_n (\omega - s)^n$ and $h(\omega)= \sum_{n=0}^{\infty}h_n (\omega-s)^n$
converge in $U$.

Our hypothesis on invertibility of $T_s\colon \xo_{\bx, s} \to \yo_{\bx, s}$ implies that
\[
\widetilde{T}_s \colon \mathcal{F}_{\bx}(\xo)/N_s(\xo)
\to \mathcal{F}_{\bx}(\yo)/N_s(\yo)
\]
is also invertible. Thus, there exists $f\in \mathcal{F}_{\bx}(\xo)$ such that
\[
\widetilde{T}_s\big(f + N_s(\xo)\big) = k + N_s(\yo).
\]
Combining with $\|\widetilde{T}_{\omega}^{-1}\| = \|T^{-1}\|_{\yo_{\bx, {\omega}} \to
\xo_{\bx, {\omega}}}$ for all $\omega \in \mathbb{A}$, we conclude that
\begin{align*}
\big\|f + N_s(\xo)\big\|_{\mathcal{F}_{\bx}(\xo)/N_s(\xo)} & =
\big\|\widetilde{T_s}^{-1}\big(\widetilde{T}_s\big(f + N_s(\xo)\big)\big)\big\|_{\mathcal{F}_{\bx}(\xo)/N_s(\xo)} \\
& \leq \|T_s^{-1}\|\,\big\|k + N_s(\yo)\big\|_{\mathcal{F}_{\bx}(\yo)/N_s(\yo)} \\
& \leq \|T_s^{-1}\|\,\|k\|_{\mathcal{F}_{\bx}(\yo)}.
\end{align*}
Now let $\varepsilon = \frac{c_1}{\|T_s^{-1}\|}-1$. Then, there exists $f_0\in N_s(\xo)$
such that
\begin{align*}
\|f-f_0\|_{\mathcal{F}_{\bx}(\xo)} & \leq (1 + \varepsilon) \big\|f + N_s(\xo)\big\|_{\mathcal{F}_{\bx}(\xo)/N_s(\xo)}  \\
& = \frac{c_1}{\|T_s^{-1}\|} \big\|f + N_s(\xo)\big\|_{\mathcal{F}_{\bx}(\xo)/N_s(\xo)} \leq c_1 \|k\|_{\mathcal{F}_{\bx}(\yo)}.
\end{align*}

Hence for $g_0 := f-f_0 \in \mathcal{F}_{\bx}(\xo)$, we have $g_0 + N_s(\xo) = f + N_s(\xo)$ and
\[
\|g_0\|_{\mathcal{F}_{\bx}(\xo)} \leq c_{1} \|k\|_{\mathcal{F}_{\bx}(\yo)}.
\]
Clearly this yields \big(by $\widetilde{T}f_0(s) = 0$ and $g_0 + N_s(\xo) = f + N_s(\xo)$\big)
\[
k - \widetilde{T}g_0 \in N_s(\yo),
\]
and
\[
\tag {$*$} \, \|\widetilde{T}g_0 - k\|_{\mathcal{F}_{\bx}(\yo)} \leq \|k\|_{\mathcal{F}_{\bx}(\yo)} \big(1 + c_1
\|\widetilde{T}\|\big).
\]
We claim that there exists $h_0 \in \mathcal{F}_{\bx}(\yo)$ such that
\[
V_{s}h_0 = k - \widetilde{T}g_0, \quad\, \|h_0\|_{\mathcal{F}_{\bx}(\yo)} \leq c \|k\|_{\mathcal{F}_{\bx}(\yo)}.
\]
To see this observe that, for all  $h\in \mathcal{F}_{\bx}(\yo)$, we have
\[
\gamma(V_s)\,\|h\|_{\mathcal{F}_{\bx}(\yo)} \leq \|V_{s}h\|_{\mathcal{F}_{\bx}(\yo)}.
\]
According to Proposition \ref{RN}, we can find \big(by $R(V_s)= N_s(\yo)$\big) $h_0\in \mathcal{F}_{\bx}(\yo)$
such that
\[
V_{s}h_0 = k  - \widetilde{T}g_0.
\]
Then by estimate $(*)$, one has
\begin{align*}
\|h_0\|_{\mathcal{F}_{\bx}(\yo)} & \leq \frac{1}{\gamma(V_s)} \|V_{s}h_0\|_{\mathcal{F}_{\bx}(\yo)} \\
& \leq \frac{c}{1 + c_1\|\widetilde{T}\|}\,\|k  - \widetilde{T}g_0\|_{\mathcal{F}_{\bx}(\yo)} \\
& \leq c \|k\|_{\mathcal{F}_{\bx}(\yo)}.
\end{align*}
In consequence, we deduce that the claim holds for $h_0$.

Similarly we find $g_1\in \mathcal{F}_{\bx}(\xo)$ and $h_1 \in \mathcal{F}_{\bx}(\yo)$ such that
\[
\widetilde{T}g_1 + V_{s}h_1 = - h_0
\]
and
\[
\|g_1\|_{\mathcal{F}_{\bx}(\xo)} \leq c_1c\|k\|_{\mathcal{F}_{\bx}(\yo)}, \quad\,\,
\|h_1\|_{\mathcal{F}_{\bx}(\yo)} \leq c^2 \|k\|_{\mathcal{F}_{\bx}(\yo)}.
\]
Now continuing the process, we construct sequences $\{g_n\}\subset \mathcal{F}_{\bx}(\xo)$,
$\{h_n\} \subset \mathcal{F}_{\bx}(\yo)$ such that, for each $n\in \mathbb{N}$ we have
\[
\|g_n\|_{\mathcal{F}_{\bx}(\xo)} \leq c_1 c^{n-1} \|k\|_{\mathcal{F}_{\bx}(\yo)},  \quad\,\,
\|h_n\|_{\mathcal{F}_{\bx}(\yo)} \leq c^n \|k\|_{\mathcal{F}_{\bx}(\yo)}.
\]
This implies that the functions $g\colon U \to \mathcal{F}_{\bx}(\xo)$ and $h\colon U \to \mathcal{F}_{\bx}(\yo)$,
given by
\[
g(\omega) = \sum_{n=0}^{\infty} g_n (\omega - s)^n, \quad\, h(\omega)=
\sum_{n=0}^{\infty} h_n (\omega-s)^n, \quad\, \omega\in U
\]
are analytic in $U$ and satisfy the desired statement.
\end{proof}

\vspace{2 mm}
Now we are ready to proof of Theorem $\ref{uinverse}$.

\begin{proof}
For a fixed $y\in Y_0 \cap Y_1$ let $k$ be a~constant function given by $k(z)= y$ for all $z\in \mathbb{A}$.
Since $k\in \mathcal{F}_{\bx}(\yo)$, it follows from Lemma \ref{onto} that there exist an open neighborhood
$U\subset \mathbb{A}$ of $s$ and analytic functions $g\colon U \to \mathcal{F}_{\bx}(\xo)$, $h\colon U \to
\mathcal{F}_{\bx}(\yo)$ such that, for all $\omega \in U$ and all $z\in \mathbb{A}$, we have
\[
T(g(\omega)(z)) + (\omega - z)\,h(\omega)(z) = y.
\]
Define a~function $\widetilde{g}\colon U \to X_0 + X_1$ by
\[
\widetilde{g}(\omega) = g(\omega)(\omega), \quad\, \omega\in U.
\]
Then $\widetilde{g}$ is analytic in $U$ and $T(\widetilde{g}(\omega)) = y$ by the above formula. Further,
$g(\omega) \in \mathcal{F}_{\bx}(\xo)$ implies $\widetilde{g}(\omega)\in \xo_{\bx, \omega}= \xo_{\bx, |\omega|}$
(by Lemma \ref{rotation} (iii)). Since $T_{\omega} \colon \xo_{\bx, |\omega|} \to \yo_{\bx, |\omega|}$ is invertible
for all $\omega \in U$,
\[
\widetilde{g}(\omega) = T_{|\omega|}^{-1}(y), \quad\, \omega\in U.
\]
In particular this implies that the analytic function $\widetilde{g}$ is constant on an open arc of the circle
with the center at $0$ and radius $|s|$ which is contained in $U$. Thus $\widetilde{g}$ is constant in $U$ by
the uniqueness theorem. Hence $T_{\omega}^{-1}y$ is independent of $\omega\in U$. To finish the proof it is enough
to combine an obvious inequality,
\[
|e^{\theta} - e^{\theta_{*}}| \leq e |\theta - \theta_{*}|, \quad\, \theta, \,\theta_{*} \in (0, 1)
\]
with norm estimates of inverse operators given in Theorem \ref{inverse}.
\end{proof}

\section{The uniqueness of inverses on intersection of interpolated Banach spaces}

The main result of Section 3, Theorem \ref{uinverse} motivates a natural question related to uniqueness of inverses
between interpolated spaces in abstract setting. Before we formulate a question we introduce a~key definition.

A family $\{F_\theta\}_{\theta \in (0, 1)}$ of interpolation functors is said to be \emph{stable} if for
any Banach couples $\ao = (A_0, A_1)$ and $\bo = (B_0, B_1)$ and for every operator $S\colon \ao \to \bo$
such that the restriction $S_{\theta_{*}}$ of $S$ to $F_{\theta_{*}}(\ao)$ is invertible for some
$\theta_{*} \in (0, 1)$ there exists $\varepsilon >0$ such that, for any $\theta \in
I(\theta_{*})= (\theta_{*}-\varepsilon, \theta_{*} + \varepsilon)$, we have
\begin{itemize}
\item[{\rm(i)}] $S_{\theta }\colon F_{\theta}(\ao) \to F_{\theta}(\bo)$ are invertible operators{\rm;}
\item[{\rm(ii)}] $S_{\theta }^{-1}\colon F_{\theta}(\bo) \to F_{\theta}(\ao)$ agrees with
$S_{\theta_{*} }^{-1}\colon F_{\theta_{*}}(\bo) \to F_{\theta_{*}}(\ao)$ on $B_0 \cap B_1$,
i.e., $S^{-1}_{\theta}y = S^{-1}_{\theta_{*}}y$ for all $y\in B_0\cap B_{1}${\rm;}
\item[{\rm(iii)}] $\sup_{\theta \in I(\theta_{*})} \vert \vert S_{\theta}^{-1} \vert \vert_{F_{\theta }(\bo) \to F_{\theta }(\ao)}
\leq C \vert \vert S_{\theta_{*}}^{-1} \vert \vert _{F_{\theta_{*} }(\bo) \to F_{\theta_{*} }(\ao)}$ for some
$C=C(\theta_{*})$.
\end{itemize}

An immediate consequence of Theorem \ref{uinverse} is the following.

\begin{corollary}
\label{stableex}
If $\bx=(\mathcal{X}_0, \mathcal{X}_1)$ is a Banach couple of translation and rotation invariant pseudolattices, then
the following family of interpolation functors $\{F_{\theta } \}_{\theta \in (0, 1)}$ is stable, where
\[
F_{\theta }(A_0, A_1) \cong (A_0, A_1)_{\vec{\chi}, e^\theta }
\]
for any Banach couple $(A_0, A_1)$.
\end{corollary}

Let $\{F_{\theta}\}_{\theta \in (0, 1)}$ be a stable family of interpolation functors and
$T \colon (X_0, X_1) \to (Y_0, Y_1)$ be a~bounded linear operator from a~Banach couple $\xo=(X_0, X_1)$ to a~Banach
couple $\yo =(Y_0, Y_1)$. Then the set of all $\theta \in (0, 1)$ for which $T\colon F_{\theta}(X_0, X_1))\rightarrow
F_{\theta }(Y_0, Y_1)$ is invertible, is open, so it is a~union of open disjoint intervals. These intervals we will call
\emph{intervals of invertibility} of $T$~with respect to the family $\{F_{\theta}\}_{\theta \in (0, 1)}$.

Let $I \subset (0, 1)$ be any interval of invertibility of $T$. In this section we are interested in the following question:
is it true that for any $\theta$, $\theta' \in I$\ the inverses $T_{\theta}^{-1}$ and $T_{\theta'}^{-1}$ agree on
$F_{\theta}(Y_0, Y_1) \cap F_{\theta'}(Y_0, Y_1)$? We point out that this problem is very important for PDEs (see, for
example, discussions in \cite{KMM}).

We will often use the following simple proposition.

\begin{proposition}
\label{kernel}
Let $\ao= (A_0, A_1)$ and $\bo = (B_0, B_1)$ be Banach couples and let $T\colon \ao \to \bo$ be an invertible operator.
Then, the following conditions are equivalent{\rm:}
\begin{itemize}
\item[{\rm(i)}] $(T|_{A_0})^{-1}y = (T|_{A_1})^{-1}y, \quad\, \text {for all} \quad y\in B_0 \cap B_{1} $\,{\rm;}
\item[{\rm(ii)}] $T\colon A_0 + A_1 \to B_0 + B_1$ is invertible\,{\rm;}
\item[{\rm(iii)}] For any interpolation functor $G$ an operator $T|_{G(\ao)}\colon G(\ao)\to G(\bo)$ is invertible.
\end{itemize}
\end{proposition}

\begin{proof} \rm(i) $\Rightarrow $ \rm(ii). Since $T\colon \ao \to \bo$ is invertible hence $T\colon A_0 +A_1 \to B_0 + B_1$
is surjective and therefore it is enough to prove that $T\colon A_0 +A_1 \to B_0 + B_1$ is injective. Let $x\in A_0+A_1$
and $Tx=0$. Then there exists a decomposition $x=x_0+x_1, \ x_0\in A_0, \ x_1\in A_1$. From $Tx=0$ it follows that
$y=Tx_0=-Tx_1\in B_0\cap B_1$. Then from \rm(i), we get that $x_0=-x_1$, and whence $x=0$.

\rm(ii) $\Rightarrow $ \rm(iii). Let denote by $T^{-1}$ the inverse operator to $T\colon A_0 +A_1 \to B_0 + B_1$. Clearly
$T^{-1}$ is a~bounded linear operator from $ (B_0, B_1)$ to  $(A_0, A_1)$ and so $T^{-1}|_{G(\bo)}
\colon G(\bo)\to G(\ao)$ is an inverse operator to $T|_{G(\ao)}\colon G(\ao)\to G(\bo)$.

The same arguments show that \rm(ii) $\Rightarrow$ \rm(i). Since $G(\vec A)=A_0+A_1$ is an interpolation functor, the
implication \rm(iii) $\Rightarrow $ \rm(ii) follows.
\end{proof}

\vspace{1.5 mm}

Now we are ready to state and prove the following result.

\begin{theorem}\label{realstable} Let  $T\colon (X_0, X_1)\rightarrow (Y_0, Y_1)$ be a linear bounded operator and
$I\subset (0, 1)$ be an interval of invertibility of \,$T$ with respect to the stable family of interpolation functors
$\{F_{\theta }\}_{\theta \in (0, 1)}$. If $Y_0\cap Y_1$ is dense in $F_{\alpha }(\yo)\cap F_{\beta }(\yo)$ for all $\alpha$,
$\beta \in (0, 1)$, then for any $\theta_0$, $\theta_1 \in I$ the inverse operators $T^{-1}_{\theta_0}$ and $T^{-1}_{\theta_1}$
agree on $F_{\theta_0}(\yo)\cap F_{\theta_1}(\yo)$.
\end{theorem}

\begin{proof} Since $Y_0\cap Y_1$ is dense in $F_{\theta_0}(\yo)\cap F_{\theta_1 }(\yo)$, for any $y\in F_{\theta_0}(\yo)
\cap F_{\theta_1}(\yo)$ there is a~sequence $\{y_n\} \subset Y_0\cap Y_1$ which converges to $y$ in $F_{\theta_0}(\yo)$ and
$F_{\theta_1}(\yo)$. Our hypothesis that the family of functors $\{F_{\theta }\}_{\theta \in (0, 1)}$ is stable implies that
$T^{-1}_{\theta_0}y_n =T^{-1}_{\theta_1}y_n$. Clearly that $x_{n}:=T^{-1}_{\theta_0}y_n=T^{-1}_{\theta_1}y_n \to T^{-1}_{\theta_0}y$
in $F_{\theta_0}(\xo)$. We also have that $x_n \to T^{-1}_{\theta_1}y$ in $F_{\theta_1}(\xo)$. In consequence the sequence $(x_n)$
converges to elements $T^{-1}_{\theta_0}y$ and $T^{-1}_{\theta_1}y$ in $X_0+X_1$. Thus $T^{-1}_{\theta_0}y=T^{-1}_{\theta_1}y$ as
required.
\end{proof}

\begin{remark} The condition that $Y_0\cap Y_1$ is dense in $F_{\alpha }(\yo)\cap F_{\beta }(\yo)$ for all $\alpha , \beta \in (0, 1)$
is rather restrictive. For example, it is not true for $F_{\theta }(Y_0, Y_1)=(Y_0, Y_1)_{\theta , \infty}$ for any
non-trivial Banach couple $(Y_0, Y_1)$, i.e., such that $Y_0 \cap Y_1$ is not closed in $Y_0 + Y_1$. However if
$\{F_{\theta}\}_{\theta\in (0, 1)}$ is a~family of regular $K$-functors, then from Remark {\rm{3.6.5}} in \cite{BK} easily follows that
this condition is fulfilled. In particular it is true for families of functors given by
\[
F_{\theta}(Y_0, Y_1)=(Y_0, Y_1)^{\circ}_{\theta , \infty} \quad\,  \text{and \,\,\,$F_{\theta}(Y_0, Y_1)=(Y_0, Y_1)_{\theta , q}$
\, with $q<\infty$.}
\]
\end{remark}

In the next proposition we show that under approximation hypothesis on $(Y_0, Y_1)$ the density condition required in
Theorem \ref{realstable} holds. Let us remind that the functor $F_{\theta }$ is said to be of {\it type} $\theta $ if
for any Banach couple $\ao =(A_0, A_1)$, we have continuous inclusions
\[
\ao_{\theta , 1}\hookrightarrow F_{\theta }(\ao)\hookrightarrow  \ao_{\theta , \infty}.
\]

\begin{proposition}
Assume that a Banach couple $\yo= (Y_0, Y_1)$ satisfies the following approximation condition{\rm:} there exists a~sequence
$\{P_n\}_{n=1}^\infty$ of linear operators $P_n\colon Y_0 + Y_1 \to Y_0 \cap Y_1$ such that
$\sup\limits_{n\geq 1} \|P_n\|_{\vec Y \to \vec Y}<\infty$ and $\|P_{n}y - y\|_{Y_0} \to 0$ as $n\to \infty$. Then,  for
any pair of regular interpolation functors $F_{\theta_1}$ and $F_{\theta_2}$ of type $\theta_1$ and $\theta_2$, respectively,
we have that $Y_0 \cap Y_1$ is dense in $F_{\theta_0}(\vec Y) \cap F_{\theta_1}(\vec Y)$.
\end{proposition}

\begin{proof}
At first we note that there exists a constant $C>0$ such that, for each $j\in \{0, 1\}$, we have
\[
\|y\|_{F_{\theta_j}} \leq C \|y\|_{Y_0}^{1-\theta_j}\,\|y\|_{Y_1}^{\theta_j}, \quad\, y\in Y_0 \cap Y_1.
\]
Hence, we get that for all $y \in Y_0 \cap Y_1$ and each $j\in\{0, 1\}$,
\[
\lim_{n\to \infty} \|P_n(y) - y\|_{F_{\theta_j}(\vec Y)} = 0.
\]
By interpolation property, it follows that $\sup\limits_{n\geq 1}\|P_n\|_{F_{\theta_j}(\vec Y)\to F_{\theta_j}(\vec Y)} <\infty$.
Since the functors are regular, we deduce that
\[
\lim_{n\to \infty} \|P_n(y) - y\|_{F_{\theta_0}(\vec Y) \cap F_{\theta_1}(\vec Y)} = 0
\]
for every $y \in F_{\theta_0}(\vec Y) \cap F_{\theta_1}(\vec Y)$, as required.
\end{proof}

\noindent
We note that Lions \cite{Lions} showed that a very wide class of Banach couples satisfy the approximation condition used
in the above proposition.

\vspace{1.5 mm}

We will say that a family of interpolation functors $\{F_{\theta }\}_{\theta \in (0, 1)}$ satisfies the global
$(\Delta)$-condition if for any Banach couple $\vec A=(A_0, A_1)$ and for any $\theta_0$, $\theta_1$ with
$0<\theta_0<\theta_1<1$, we have continuous inclusions
\begin{equation}
\label{embeddings}
F_{\theta_0}(\vec A)\cap F_{\theta_1 }(\ao)\hookrightarrow \bigcap \limits_{\theta_0<\theta <\theta_1}F_{\theta }(\ao)
\hookrightarrow (F_{\theta_0}(\ao))^{c} \cap (F_{\theta_1 }(\ao))^{c} ,
\end{equation}
where the norm in $\bigcap \limits_{\theta_0<\theta <\theta_1}F_{\theta }(\ao)$ is given by
\begin{equation}
\Arrowvert a \Arrowvert _{\bigcap \limits_{\theta_0<\theta <\theta_1}F_{\theta }(\ao)}=\sup\limits_{\theta_0<\theta
<\theta_1}\Arrowvert a \Arrowvert _{F_{\theta }(\ao)}. 
\end{equation}
and the Gagliardo completion $(F_{\theta_i}(\ao))^c, \ j\in \{0, 1\}$ is taken with respect to the sum
$F_{\theta_0}(\ao) + F_{\theta_1}(\ao)$.

In what follows we will use the following obvious observation.

\begin{proposition}
\label{stablefunctors}
Let $\{F_\theta\}_{\theta \in (0, 1)}$ and  $\{G_\theta\}_{\theta \in (0, 1)}$ be families of interpolation functors.
Suppose that there exist positive  functions $C_1$, $C_2$ defined on $(0, 1)$ which are bounded on every compact
subinterval of $(0, 1)$ and such that $F_\theta(\xo) = G_\theta(\xo)$ with
\[
C_1(\theta)\,\|\cdot\|_{G_\theta(\xo)} \leq \|\cdot\|_{F_\theta(\xo)} \leq C_2(\theta)\,\|\cdot\|_{G_\theta(\xo)}
\]
for every Banach couple $\xo$ and all $\theta \in (0, 1)$. Then
\begin{itemize}
\item[{\rm(i)}] The family $\{F_\theta\}_{\theta \in (0, 1)}$ is stable if and only if
the family $\{G_\theta\}_{\theta \in (0, 1)}$ is stable.
\item[{\rm(ii)}] The family $\{F_{\theta }\}_{\theta \in (0, 1)}$ satisfies the
global $(\Delta)$-condition if, and only if, the family $\{G_{\theta }\}_{\theta \in (0, 1)}$ satisfies the
global $(\Delta)$-condition.
\end{itemize}
\end{proposition}

To state and prove the theorem on stability of inverses on interpolated spaces we need one more definition. We say
that a~family of interpolation functors $\{F_{\theta }\}_{\theta \in (0, 1)}$ satisfies the reiteration condition
if for any Banach couple $\ao=(A_0, A_1)$ and for any $\theta_0$, $\theta_1$, $\lambda \in (0, 1)$ we have
\begin{align*}
\label{reiteration}
F_{\lambda } (F_{\theta _0}(\ao), F_{\theta _1}(\ao)) = F_{(1-\lambda )\theta _0+\lambda \theta _1}(\ao).
\end{align*}

\vspace{2 mm}

\begin{theorem} \label{maintheorem} Let  $T\colon (X_0, X_1)\rightarrow (Y_0, Y_1)$ be a linear bounded operator and
$I\subset (0, 1)$ be an interval of invertibility of \,$T$ with respect to the stable family of interpolation functors
$\{F_{\theta }\}_{\theta \in (0, 1)}$. If $\{F_{\theta }\}_{\theta \in (0, 1)}$ satisfies the global $(\Delta)$--\!and
reiteration conditions, then for any $\theta_0, \theta_1 \in I$, the inverse operators $T^{-1}_{\theta_0}$ and
$T^{-1}_{\theta_1}$ agree on $F_{\theta_0}(\yo)\cap F_{\theta_1 }(\yo)$.
\end{theorem}

\begin{proof} We deduce from Proposition \ref{kernel} that it is enough to prove invertibility of the operator
\[
T\colon  F_{\theta _0}(\xo)+F_{\theta _1}(\xo)\rightarrow F_{\theta _0}(\yo)+F_{\theta _1}(\yo) 
\]
Since $T\colon F_{\theta_{j}}(X_0,X_1)\rightarrow F_{\theta_{j}}(Y_0,Y_1)$ for $j \in\{0, 1\}$, is invertible,
$T\colon F_{\theta _0}(X_0, X_1)+F_{\theta_1}(X_0, X_1)\rightarrow F_{\theta_0}(Y_0, Y_1)+F_{\theta _1}(Y_0, Y_1)$
is surjective. Then, it is enough to prove that $T\colon F_{\theta_0}(X_0, X_1)+F_{\theta_1}(X_0, X_1)\rightarrow
F_{\theta _0}(Y_0, Y_1)+F_{\theta _1}(Y_0, Y_1)$ is injective. Let $x\in F_{\theta _0}(X_0, X_1)+F_{\theta _1}(X_0, X_1)$
and $Tx=0$. Then $x=x_0+x_1$, where $x_j\in F_{\theta_j}(X_0, X_1), \j\in \{0, 1\}$. From $Tx=0$, we have
\[
y= Tx_0 = - Tx_1\in F_{\theta_0}(Y_0, Y_1)\cap F_{\theta_1}(Y_0, Y_1).
\]
Since $\{F_{\theta }\}_{\theta \in (0, 1)}$ satisfies the global $(\Delta)$-condition,  $y\in F_{\theta }(\yo)$ for all
$\theta \in (\theta_0, \theta_1)$ and
\begin{align*}
\tag {$*$} \, \Arrowvert y \Arrowvert _{\bigcap \limits_{\theta_0<\theta <\theta_1}F_{\theta }(\yo}\leq C \Arrowvert y
\Arrowvert _{F_{\theta _0}(\yo)\cap F_{\theta _1}(\yo)}.
\end{align*}
Fix $\theta_0$, $\theta_1\in I$. Then for any $\theta \in (\theta_0, \theta_1)$ the operator $T_{\theta }=T|_{F_{\theta }(\xo)}$
is invertible and so $x_{\theta}:= T^{-1}_{\theta}y$ is well defined.

We claim that $x_{\theta}$ does not depend on $\theta \in (\theta_0, \theta_1)$. To see this let us consider the couples
\[
(\widetilde{X}_0, \widetilde{X}_1)=(F_{\theta _0}(\xo), F_{\theta _1}(\xo)), \ \ (\widetilde{Y}_0, \widetilde{Y}_1)
=(F_{\theta _0}(\yo), F_{\theta _1}(\yo)).
\]
Let $\widetilde{T}_{\lambda }$ be the restriction of the operator $T\colon (X_0, X_1) \to (Y_0, Y_1)$ to
$F_{\lambda }(\widetilde{X}_0, \widetilde{X}_1)$. Since $\theta =(1-\lambda )\theta_0+\lambda \theta_1$ for some $\lambda \in (0, 1)$,
it follows from our hypothesis (on the reiteration condition) that $\widetilde{T}_{\lambda }x=T_{\theta}x$, for any
$x\in  F_{\lambda} (\widetilde  X_0, \tilde X_1)=F_{\theta }(X_0, X_1)$. Hence
\[
\tilde{T}_{\lambda }\colon F_{\lambda }(\widetilde{X}_0, \widetilde{X}_1)\rightarrow  F_{\lambda }(\widetilde{Y}_0, \widetilde{Y}_1)
\]
is invertible for all $\lambda \in (0, 1)$.

If $\lambda_1$, $\lambda_2 \in (0, 1)$, then the stability of the family $F_{\theta }$ combined with the compactness of the subinterval
$[\lambda_1, \lambda_2]$ of $(0, 1)$ yields that inverse operators
\[
\widetilde{T}_{\lambda_1}^{-1}\colon F_{\lambda_1}(\widetilde{Y}_0, \widetilde{Y}_1)\rightarrow F_{\lambda_1}(\widetilde{X}_0,
\widetilde{X}_1) \ \ \text{and} \  \ \widetilde{T}_{\lambda_2}^{-1}\colon F_{\lambda_2}(\tilde{Y}_0, \widetilde{Y}_1)\rightarrow F_{\lambda_2}(\tilde{X}_0, \widetilde{X}_1)
\]
agree on $\widetilde{Y}_0\cap \widetilde{Y}_1=F_{\theta_0}(Y_0, Y_1)\cap F_{\theta_1}(Y_0, Y_1)$. Hence the element
$x_{\theta }=T^{-1}_{\theta }y=\widetilde{T}_{\lambda}^{-1}y$ is independent of $\theta \in (\theta_0, \theta_1)$ and we denote it
by $\bar{x}$. Moreover,  the element $\bar{x}$ belongs to the set $\bigcap \limits_{\theta_0<\theta  <\theta_1} F_{\theta }(\xo)$ and
$\Arrowvert \bar{x} \Arrowvert _{\bigcap \limits_{\theta_0<\theta  <\theta_1} F_{\theta }(\xo)}<\infty$. Indeed, from invertibility
of the operator $T$ on the whole interval $I$, stability of the family $F_{\theta }$ and compactness of the interval $[\theta_0, \theta_1]$,
we get that
\[
\sup\limits_{\theta_0<\theta<\theta_1}\Arrowvert T^{-1}_{\theta }  \Arrowvert _{F_{\theta}(Y_0, Y_1)\rightarrow F_{\theta }(X_0, X_1)} < \infty.
\]
Hence from the shown above estimate $(*)$, we obtain
\[
\Arrowvert \bar{x} \Arrowvert _{\bigcap \limits_{\theta_0<\theta  <\theta_1} F_{\theta }(\xo)}\leq \sup\limits_{\theta_0<\theta<\theta_1}\Arrowvert T^{-1}_{\theta }  \Arrowvert _{F_{\theta }(\yo)\rightarrow F_{\theta }(\xo)}\Arrowvert y \Arrowvert _{F_{\theta_0 }(\yo)\cap F_{\theta_1}(\yo)}<\infty.
\]
Thus using the right hand continuous inclusion in the definition of the global $(\Delta)$-condition, we conclude that
\[
\bar{x}\in (F_{\theta_0}(\xo))^c\cap (F_{\theta_1}(\xo))^c.
\]
Now to finish the proof we decompose the element $x$ as
\[
x=x_0+x_1 = (x_0 - \bar{x})+(x_1 + \bar{x}). 
\]
Since $Tx_0 = - Tx_1=T\bar{x}$, it is clear that
\[
x_0 - \bar{x} \in (F_{\theta _0}(\xo))^c \cap \ker T \ \text{and} \ x_1 + \bar{x} \in (F_{\theta_1}(\xo))^c \cap \ker T.
\]
Invertibility of the operator $T$ on $F_{\theta_j}(\xo)$ implies injectivity of $T$ on $(F_{\theta_j}(\xo))^c$
for each $j\in \{0, 1\}$. This implies that both $x_0 - \bar{x}$ and $x_1 + \bar{x}$ are equal to zero. Consequently
$x=0$ and so the operator $T\colon F_{\theta _0}(\xo) + F_{\theta _1}(\xo) \to F_{\theta _0}(\yo)+F_{\theta _1}(\yo)$
is invertible.
\end{proof}

To show applications to complex and real interpolation methods of the above results we need a lemma.

\begin{lemma}
\label{CalRealstable}
The families $\{[\,\cdot\,]_{\theta}\}_{\theta \in (0, 1)}$ of the Calde\'on functors, as well as
$\{(\,\cdot\,)_{\theta, q}\}_{\theta \in (0, 1)}$ with $1\leq q\leq \infty$ of the Lions--Peetre
interpolation functors, are both stable.
\end{lemma}

\begin{proof} Let $\{G_\theta\} := \{[\,\cdot\,]_{\theta}\}_{\theta \in (0, 1)}$. At first we note that
it is shown in \cite{Ivtsan} that for any Banach couple $(A_0, A_1)$ we have
\[
(A_0, A_1)_{\theta, e^{\theta}} \cong [A_0, A_1]_{\theta}^{\lambda}, \quad\, \theta \in (0,1),
\]
where $[A_0, A_1]_{\theta}^{\lambda}$ is the "periodic" interpolation space with $\lambda= 2\pi$. It follows
immediately from the definition of the periodic interpolation space that
\[
[A_0, A_1]_{\theta}^{2\pi} \hookrightarrow [A_0, A_1]_{\theta}
\]
with norm of the inclusion map less or equal than $1$. Analysis of the proof of Equivalence in \cite[p.~1008]{Cwikel}
shows that
\[
[A_0, A_1]_{\theta} \hookrightarrow [A_0, A_1]_{\theta}^{2\pi}, \quad\, \theta \in (0, 1)
\]
with norm of the inclusion map less or equal than $C(\theta)$. Standard calculus shows that there exists a~positive
constant $K$ independent of $\theta$ such that
\[
C(\theta) \leq \frac{K}{\theta(1-\theta)}.
\]
Altogether yields that the family $\{F_{\theta}\}:=  \{(\,\cdot\,)_{(FC, FC), e^{\theta}}\}_{\theta \in (0, 1)}$ satisfies
\[
F_{\theta}(A_0, A_1) = G_{\theta}(A_0, A_1), \quad\, \theta \in (0,1),
\]
where the constants of equivalence of norms are bounded on any compact subinterval of $(0, 1)$. To finish it is enough to
apply Corollary \ref{stableex} and Proposition \ref{stablefunctors}.

Now we consider the case $\{G_\theta\}:=\{(\,\cdot\,)_{\theta, q}\}_{\theta\in (0, 1)}$ for any fixed $1\leq q\leq \infty$.
Put $\{F_\theta\} := \{(\,\cdot\,)_{(\ell_q, \ell_q), e^\theta}\}_{\theta \in (0,1)}\big\}$. It was noticed in Section 2
that
\[
F_\theta(\ao) \cong G_\theta(\ao)
\]
for all Banach couples. It is well known that $F_{\theta}(\ao) = G_{\theta}(\ao)$ up to equivalence of norms. Standard
calculus shows there exist absolute positive constants $C_1>0$ and $C>0$, independent on $\theta$ and
$q\in [1, \infty)$, such that
\[
C_1\|\cdot\|_{F\theta(\ao)} \leq \|\cdot\|_{G_\theta(\ao)} \leq \frac{B}{\theta(1-\theta)}\|\cdot\|_{G_\theta(\ao)}.
\]
Again applying Corollary \ref{stableex} and Proposition \ref{stablefunctors} we are done.
\end{proof}

We are ready to prove the compatibility theorem for the family of real interpolation functors.

\begin{theorem}
\label{realcompatibility}
Let $1\leq q\leq \infty$ and let  $T\colon (X_0, X_1)\rightarrow (Y_0, Y_1)$ be a~linear bounded operator
and $I\subset (0, 1)$ be an interval of invertibility of \,$T$ with respect to the family
$\{(\cdot)_{\theta,q}\}_{\theta\in (0, 1)}$ of real interpolation
functors. Then for any $\theta_0$, $\theta_1 \in I$ the inverse operators $T^{-1}_{\theta_0}$ and $T^{-1}_{\theta_1}$
agree on $(Y_0, Y_1)_{\theta_0, q}\cap (Y_0, Y_1)_{\theta_1, q}$.
\end{theorem}

\begin{proof} It is well known that the family of real interpolation functors  satisfies the reiteration condition.
Moreover stability of this family follows from Lemma \ref{CalRealstable}. Thus in order to apply Theorem \ref{maintheorem},
we only need to check that this family satisfies the global ($\Delta $)-condition (\ref{embeddings}).

Since $((A_0, A_1)_{\theta, q})^{c} \cong (A_0, A_1)_{\theta , q}$), it is enough to prove that for any Banach couple
$\ao =(A_0, A_1)$, we have
\[
(A_0, A_1)_{\theta_0, q}\cap (A_0, A_1)_{\theta_1, q} =  \bigcap \limits_{\theta_0<\theta <\theta_1} (A_0, A_1)_{\theta , q}.
\]
Let $x\in (A_0, A_1)_{\theta_0, q}\cap (A_0, A_1)_{\theta_1, q}$ and $\theta_0<\theta <\theta_1$. Then
\begin{align*}
\Arrowvert x \Arrowvert _{\theta , q} & = \bigg( \int \limits_0^\infty \big(t^{-\theta }K(t,x; \ao)\big)^q \frac{dt}{t}\bigg)^{1/q} \\
& \leq \bigg(\int\limits_0^\infty \big(t^{-\theta _0}K(t,x; \ao)\big)^q \frac{dt}{t}\bigg)^{1/q}+ \bigg(\int \limits_0^\infty \big(t^{-\theta_1}
K(t,x; \ao)\big)^q \frac{dt}{t}\bigg)^{1/q} \\
& \leq 2 \max \{\Arrowvert x \Arrowvert _{\theta_0, q}, \Arrowvert x \Arrowvert_{\theta_1, q}\}.
\end{align*}
This yields $(A_0, A_1)_{\theta_0, q}\cap (A_0, A_1)_{\theta_1, q} \hookrightarrow \bigcap \limits_{\theta_0<\theta <\theta_1} (A_0, A_1)_{\theta, q}.$

Now let $x\in \bigcap \limits_{\theta_0<\theta <\theta_1} (A_0, A_1)_{\theta , q}$ and $0<a<b<\infty$. Since $\theta \in (\theta_0, \theta_1 )$
is arbitrary,
\[
\bigg(\int \limits_{a}^{b} \big(t^{-\theta_0} K(t, x; \ao)\big)^q\frac{dt}{t}\bigg)^{1/q}\leq b^{(\theta -\theta_0)}
\bigg(\int \limits_{a}^{b} \big(t^{-\theta } K(t, x; \ao))^q\frac{dt}{t}\bigg)^{1/q}.
\]
Similarly, we get that
\[
\bigg(\int \limits_{a}^{b} \big(t^{-\theta_1} K(t, x; \ao)\big)^q\frac{dt}{t}\bigg)^{1/q}\leq a^{(\theta -\theta_1)}
\bigg(\int \limits_{a}^{b} \big(t^{-\theta } K(t, x; \ao)\big)^q\frac{dt}{t}\bigg)^{1/q}.
\]
Taking in account that these inequalities are correct for any $\theta \in (\theta_0, \theta_1)$ and for arbitrary
$a$ and $b$, we get that
\[
\max \{\Arrowvert x \Arrowvert _{\theta_0, q},
\Arrowvert x \Arrowvert _{\theta_1, q}\}\leq \Arrowvert x\Arrowvert _{\bigcap \limits_{\theta_0<\theta <\theta_1} \ao_{\theta , q}}.
\]
Hence $\bigcap \limits_{\theta_0<\theta <\theta_1} (A_0, A_1)_{\theta , q}
\hookrightarrow (A_0, A_1)_{\theta_0, q}\cap (A_0, A_1)_{\theta_1, q}$. Similarly, we prove the case $p=\infty$.
\end{proof}

From Theorem \ref{maintheorem} also follows the compatibility theorem for the family $\{F_{\theta}\}_{\theta \in (0, 1)}$ of complex
interpolation functors:

\begin{equation}
\label{complexFunctor}
F_{\theta}(A_0, A_1)=[A_0, A_1]_{\theta}.
\end{equation}

\begin{theorem} \label{complexcompatibility}
Let  $T\colon (X_0, X_1)\rightarrow (Y_0, Y_1)$ be an operator between couples of complex Banach spaces and let $I\subset (0, 1)$ be an
interval of invertibility of $T$ with respect to the family of interpolation functors defined by $($\ref{complexFunctor}$)$. Then for any
$\theta_0$, $\theta_1 \in I$ the inverse operators $T^{-1}_{\theta_0}$ and $T^{-1}_{\theta_1}$ agree on $[Y_0, Y_1]_{\theta_0 }\cap [Y_0, Y_1]_{\theta_1}.$
\end{theorem}

\begin{proof} As well as in the proof of Theorem \ref{realcompatibility} it is enough to prove the global ($\Delta$)-condition for
the family $\{[\,\cdot\,]){\theta}\}$ for arbitrary Banach couple $(A_0, A_1)$:
\[
[A_0, A_1]_{\theta_0}\cap [A_0, A_1]_{\theta_1}\hookrightarrow \bigcap
\limits_{\theta_0<\theta <\theta_1} [A_0, A_1]_{\theta } \hookrightarrow ([A_0, A_1]_{\theta_0})^c\cap ([A_0, A_1]_{\theta_1})^{c},
\]
where Gagliardo completion $([A_0, A_1]_{\theta_j})^{c}$ for $j\in \{0, 1\}$ is taken with respect to the sum $[A_0, A_1]_{\theta_0}
+ [A_0, A_1]_{\theta_1}$. Since the reiteration formula
\[
\left[[A_0, A_1]_{\theta_0},  [A_0, A_1]_{\theta_1} \right]_{\lambda} = [A_0, A_1]_{(1-\lambda )\theta_0+\lambda \theta_1}
\]
holds with equality of norms for any $\theta_0, \theta_1, \lambda \in (0, 1)$ (see \cite{Cwikel}). Hence,
for any $x\in [A_0, A_1]_{\theta_0}\cap [A_0, A_1]_{\theta_1}$, we have
\[
\Arrowvert x \Arrowvert _{(1-\lambda )\theta_0+\lambda \theta_1} \leq \Arrowvert x \Arrowvert _{\theta_0}^{1-\lambda } \Arrowvert x \Arrowvert _{\theta_1}^{\lambda } \leq \max\{\Arrowvert x \Arrowvert _{\theta_0}, \Arrowvert x \Arrowvert _{\theta_1} \}.
\]
This proves that
\[
[A_0, A_1]_{\theta_0}\cap [A_0, A_1]_{\theta_1}\hookrightarrow \bigcap \limits_{\theta_0<\theta <\theta_1} [A_0, A_1]_{\theta}.
\]
The proof of Theorem 4.7.1 in \cite{BL} shows that for any $x\in [A_0, A_1]_{\theta }$
\[
\Arrowvert x\Arrowvert _{\theta , \infty} \leq \Arrowvert x\Arrowvert _{\theta }.
\]
Since
\[
\Arrowvert x\Arrowvert _{A_0}^c = \sup_{t>0} {K(t, x; \ao)} \quad\, \text{and} \ \Arrowvert x\Arrowvert_{A_1}^c 
= \sup_{t>0}{\frac{K(t, x; \ao)}{t}},
\]
we get that
\begin{align*}
\sup _{\theta_0<\theta <\theta_1}\Arrowvert x\Arrowvert_{\theta}& =\sup _{0<\lambda <1} \Arrowvert x\Arrowvert _{[\ao_{\theta_0},
\so_{\theta_1}]_\lambda }\geq \sup _{0<\lambda <1} \Arrowvert x\Arrowvert_{(\ao_{\theta_0}, \ao_{\theta_1})_\lambda , \infty }\\
& \geq \max \{\Arrowvert x\Arrowvert_{([A_0, A_1]_{\theta_0})^c}, \Arrowvert x\Arrowvert _{([A_0, A_1]_{\theta_1})^c}\},
\end{align*}
where the Gagliardo completion $([A_0, A_1]_{\theta_j})^c, \ j\in\{0, 1\}$ is taken with respect to the sum
$[A_0, A_1]_{\theta_0} + [A_0, A_1]_{\theta_1}$. Thus we conclude that the second required continuous inclusion
\[
\bigcap  \limits_{\theta_0<\theta <\theta_1} [A_0, A_1]_{\theta } \hookrightarrow
([A_0, A_1]_{\theta_0})^{c} \cap ([A_0, A_1]_{\theta_1})^{c}
\]
holds and so this completes the proof.
\end{proof}

\begin{theorem}
\label{ComplexReal} Let  $T\colon (X_0, X_1)\rightarrow (Y_0, Y_1)$ be an operator between couples of complex
Banach spaces. If $T\colon [X_0, X_1]_{\theta _{*}}\rightarrow [Y_0, Y_1]_{\theta_{*}}$ is invertible for some
$\theta_{*} \in (0, 1)$, then
\begin{align*}
T\colon (X_0, X_1)_{\theta _{*},q}\rightarrow (Y_0, Y_1)_{\theta _{*},q}
\end{align*}
is invertible for all $q\in [1, \infty]$.
\end{theorem}

\begin{proof} Let $\theta_{*} \in I$, where $I$ is an interval of invertibility of $T$ with respect to the family
of functors of complex interpolation. Then there exists $\varepsilon >0$ such that $\theta_0=\theta - \varepsilon ,
\  \theta_1=\theta + \varepsilon \in I.$ From Theorem \ref{complexcompatibility} follows that inverse operators
$T^{-1}_{\theta_0}$ and $T^{-1}_{\theta_1}$ agree on $[Y_0, Y_1]_{\theta _0}\cap [Y_0, Y_1]_{\theta _1}.$ So from
Proposition \ref{kernel} \rm(iii) we obtain invertibility of the operator
\[
T\colon \left([X_0, X_1]_{\theta_0}, [X_0, X_1]_{\theta_1}\right)_{\frac{1}{2}, q}\rightarrow
([Y_0, Y_1]_{\theta_0}, [Y_0, Y_1]_{\theta_1})_{\frac{1}{2}, q}.
\]
To complete the proof it remains to note that
\[
([X_0, X_1]_{\theta_0}, [X_0, X_1]_{\theta_1})_{\frac{1}{2}, q}=(X_0, X_1)_{\theta_{*}, q} 
\]
and
\[
([Y_0, Y_1]_{\theta_0}, [Y_0, Y_1]_{\theta_1})_{\frac{1}{2}, q}=(Y_0, Y_1)_{\theta _{*},q}.
\]
\end{proof}

We conclude with the following result about the connections between spectrum of interpolated operators. The result
is an immediate consequence of Theorem \ref{ComplexReal}.

\begin{theorem}
Let $\bx = (\mathcal{X}_0, \mathcal{X}_1)$ be a~Banach couple of translation and rotation invariant pseudolattices
and let the family $\{F_\theta\} := \{(\,\cdot\,)_{\bx, e^\theta}\}_{\theta\in (0, 1)}$ be such the reiteration condition
holds for a~complex Banach couple $(X_0, X_1)$. If $\{F_\theta\}$ satisfies a~global $(\Delta)$-condition for
$(X_0, X_1)$ then, for any operator $T\colon \xo \to \xo$ and all $q\in [1, \infty]$,
we have
\[
\sigma\big(T, \xo_{\theta, p}\big) \subset \sigma\big(T, F_\theta(\xo)\big).
\]
\end{theorem}

As a consequence, we obtain the following corollary.

\begin{corollary} Let $(X_0, X_1)$ be a couple of complex Banach spaces. Then, for any operator
$T\colon (X_0, X_1)\rightarrow (X_0, X_1)$ and for all $q\in [1, \infty]$, we have
\[
\sigma(T, \xo_{\theta, q}) \subset \sigma(T, [\xo]_{\theta}).
\]
\end{corollary}

We conclude with the following remark that Albrecht and M\"uller gave an example of a~ Banach couple $\xo$
and and operator $\colon \xo\to \xo$ for which $\sigma(T, \xo_{\theta, 1}) \neq \sigma(T, [\xo]_{\theta})$
(see \cite[Example 12]{AM}).

\vspace{2 mm}

\section{Order isomorphisms between Calder\'on spaces}

Throughout this section $(\Omega, \Sigma, \mu)$ denotes a~$\sigma$-finite measure space. The symbol
$L^0(\mu) := L^0(\Omega, \Sigma,\mu)$ stands for the space of (equivalence classes of $\mu$-a.e.~equal)
real-valued measurable functions on $\Omega$ with the topology of convergence in measure on $\mu$-finite
sets. As usual the order $|f|\leq |g|$ means that $|f(t)| \leq |g(t)|$ for $\mu$-almost all $t \in \Omega$.

If a~Banach space $X\subset L^{0}(\mu)$ contains an element which is strictly positive $\mu$-a.e.~on $\Omega$
and $X$ is solid (meaning that $f \in X$ with $\|f\|_X \leq \|g\|_X$ whenever $|f| \leq |g|$ with
$f\in L^{0}(\mu)$ and $g\in X$), then $X$ is said to be a~\emph{Banach lattice} on $(\Omega, \Sigma, \mu)$).
A~Banach lattice $X$ is said to have the \emph{Fatou property}, if for any sequence $\{f_n\}$ of non-negative
elements from $X$ such that $f_n\uparrow f$ for $f\in L^0(\Omega)$ and
$\sup \bigl\{\|f_n\|_{X};\, n\in \mathbb{N}\bigr\} <\infty$, one has $f\in X$ and $\|f_n\|_X \uparrow \|f\|_X$.

Let $X$ and $Y$ be Banach lattices. A linear operator $T\colon X \to Y$ is said to be \emph{positive} (resp., homomorphism)
if $Tx\geq 0$ whenever $x\geq 0$ (resp., $Tx \wedge Ty = 0$ whenever $x\wedge y =0$). A~homomorphism which is additionally
a~bijection is called an \emph{order isomorphism}. It is well known that a~linear bijection $T\colon X \to Y$ is
an order isomorphism if and only if $T$ and $T^{-1}$ are both positive (see \cite[Theorem 7.3]{AB}).

This section elaborates on an unpublished result of Milman \cite{Mi} on a~strong variant of Shnieberg result that states
that, under some mild conditions, and in the context of Banach lattices, invertibility of a~bounded positive operator
at one point of the scale of Calder\'on space for Banach function lattices implies invertibility at all points in the
interior scale. Combining with our previous results we obtain a~variant of this result for the classical real interpolation
spaces between Banach lattices.

We recall that the \emph{Calder\'on product} $X_{0}^{1-\theta}X_{1}^{\theta}$ defined for any couple $(X_0, X_1)$ of Banach lattices on
a~measure space $(\Omega, \Sigma, \mu)$  consists of all $f\in L^{0}(\mu)$ such that $|f| \leq \lambda \,|f_0|^{1-\theta}
|f_1|^{\theta}$ $\mu$-a.e.\ for some $\lambda > 0$ and $f_j \in X_j$ with $\|f_j\|_{X_j} \leq 1$, $j\in\{0, 1\}$. It is well
known (see \cite{Ca1}) that $X_{0}^{1-\theta}X_{1}^{\theta}$ is a~Banach lattice endowed with the norm
\[
\|f\| = \inf \big\{\lambda >0; \,|f| \leq \lambda
\,|f_0|^{1-\theta} |f_1|^{\theta} ,\, \|f_0\|_{X_0} \leq 1, \,\|f_1\|_{X_1} \leq 1\big\}.
\]
In what follows for simplicity of notation, we also write for short $X_{\theta}$ instead of $X_0^{1-\theta}X_1^{\theta}$.

For the reader's convenience, we include the proof of the mentioned above result.

\vspace{ 2 mm}

\begin{theorem}
\label{coneemb} Let $T\colon(X_{0}, X_{1}) \to(Y_{0}, Y_{1})$ be a positive operator between couples of Banach lattices
with the Fatou property. Assume that $T\colon X_{0}^{1-\theta_{0}}X_{1}^{\theta_{0}}\to Y_{0}^{1-\theta_{0}}Y_{1}^
{\theta_{0}}$ is an order isomorphism for some $\theta_{0} \in(0, 1)$. Then $T\colon X_{0}^{1-\theta_{1}}X_{1}^{\theta_{1}}
\to Y_{0}^{1-\theta_{1}}Y_{1}^{\theta_{1}}$ is an order isomorphism for all $\theta_{1} \in(0, 1)$.
\end{theorem}

\vspace{2 mm}

\begin{proof}
Notice that for any couple $(E_0, E_!)$ of Banach lattices with the Fatou property and for every $\theta\in(0, 1)$,
$E_0^{1-\theta}E_1^{\theta}$ is a~Banach lattice with the Fatou property (see \cite{Lo78}). Thus by use of
extrapolation formula of Cwikel--Nilsson \cite[Theorem 3.5]{CN}, we have
\[
\|f\|_{E_0^{1-\theta}E_1^{\theta}} = \sup\Big\{\big\||g|^{1-\alpha}|f|^{\alpha}\big\|_{E_{0}^{1-\alpha}
(E_0^{1-\theta}E_1^{\theta})^{\alpha}}^{1/\alpha};\, \|g\|_{E_0} \leq1 \Big\}, \quad\ \alpha, \theta\in(0, 1).
\]
Since Calder\'on construction is an interpolation method for positive operators,
\[
T\colon X_0^{1-\theta}X_1^{\theta}\to Y_0^{1-\theta}Y_1^{\theta}, \quad\, \theta\in(0, 1).
\]
Suppose that for some $C>0$ and all $f \geq0$ in $X_{\theta_0}$ we have
\[
\tag {$*$} C \|f\|_{X_{\theta_0}} \leq \|Tf\|_{Y_{\theta_0}}.
\]
We will use the following easily verified reiteration formula true for an arbitrary couple of Banach lattices,
which is true for all $\alpha$, $\theta_0$ and $\theta_1$ in $[0,1]$
\[
(X_0^{1-\theta_0}X_1^{\theta_0})^{1-\alpha}(X_0^{1-\theta_1}X_1^{\theta_1})^{\alpha}
= X_0^{1-\beta}X_1^{\beta},
\]
where $\beta= (1- \alpha)\theta_0 + \alpha\theta_1$. We also require the following property of
any positive operator $P\colon X \to Y$ between Banach lattices which says that if
$0\leq x, y\in X$ and $\theta\in(0, 1)$, then (see, e.g., \cite[p.~55]{LT})
\[
P(x^{1-\theta} y^{\theta}) \leq P(x)^{1- \theta} P(y)^{\theta}.
\]
We may assume without loss of generality that $\|T\|_{X_j \to Y_j} \leq1$ for $j\in \{0, 1\}$ and also that
$0<\theta_0 <\theta_1<1$. Thus, we can find $\alpha\in(0, 1)$ such that $\theta_0 = \alpha\theta_1$.
Suppose $f\in X_{\theta_1}$ is nonnegative. Combining Cwikel--Nilsson formula shown above with
the mentioned property of positive operators and our hypothesis we obtain
\begin{align*}
\|Tf \|_{Y_{\theta_1}} &=
\sup\Big\{\big\|\,|g|^{1-\alpha}(Tf)^{\alpha}\big\|_{Y_{\theta_0}}^{1/\alpha}\,;\,\, \|g\|_{Y_0} \leq1\Big\} \\
& \geq\sup\Big\{\big\|(T|x_0|)^{1-\alpha}(Tf)^{\alpha}\big\|_{Y_{\theta_0}}^{1/\alpha}\,; \,\, \|x_0\|_{X_0} \leq1\Big\} \\
& \geq\sup\Big\{\big\|T(|x_0|^{1- \alpha}f^{\alpha})\big\|_{Y_{\theta_0}}^{1/\alpha}; \,\, \|x_0\|_{X_0} \leq 1\Big\}.
\end{align*}
In consequence, our hypothesis $(*)$ on $T$ and the mentioned extrapolation formula yield the required estimate
\begin{align*}
\big\|Tf\big\|_{Y_{\theta_1}} & \geq C^{1/\alpha}\, \sup\Big\{\big\||x_0|^{1- \alpha}
f^{\alpha}\big\|_{Y_{\theta_0}}^{1/\alpha};\, \|x_0\|_{X_0} \leq1\Big\} \\
& = C^{1/\alpha}\, \big\|f\big\|_{X_{\theta_1}}
\end{align*}
and this completes the proof.
\end{proof}

In the sequel when the complex methods are applied to a~couple $(X_0, X_1)$ of Banach lattices, we
mean that $X_j:=X_j(\mathbb{C})$ is a~\emph{complexification} of $X_j$ for $j\in \{0, 1\}$. If $X$ is
an intermediate Banach space with respect to a~couple $\xo=(X_0, X_1)$, we let $X^{\circ}$ be the
closed hull of $X_0 \cap X_1$ in $X$.

We conclude with the following result.

\begin{theorem}
Let $\xo=(X_0, X_1)$ and $\yo = (Y_0, Y_1)$ be couples of regular Banach lattices with the Fatou property
and let $T\colon X_0 + X_1\to Y_0 + Y_1$ be a positive operator. If \,$T\colon X_0^{1-\theta_{*}} X_1^{1-\theta_{*}}
\to  Y_0^{1-\theta_{*}} Y_1^{\theta_{*}}$ is an order isomorphism for some $\theta_{*} \in (0, 1)$, then
\[
T\colon X_{0}^{1-\theta} X_{1}^{\theta} \to  Y_{0}^{1-\theta} Y_{1}^{\theta}, \quad\,
T\colon (X_0, X_1)_{\theta, p} \to  (Y_0, Y_1)_{\theta, p}
\]
are order isomorphisms for all $\theta \in (0, 1)$, $p\in [1, \infty]$.
\end{theorem}

\begin{proof}
Since the couples are regular, we have that $X_0 \cap X_1$ is dense in $X_{\alpha}$ and $Y_0\cap Y_1$ is dense
in $Y_{\alpha}$ for all $\alpha \in (0, 1)$. Thus, it follows from \cite{Nilsson} that the following formulas
hold within equivalence of norms
\[
\langle X_0, X_1\rangle_{\theta_{*}} = \big(X_0^{1-\theta_{*}} X_1^{\theta_{*}}\big)^{\circ}
= X_0^{1-\theta_{*}} X_1^{\theta_{*}},
\]
and similarly,
\[
\langle Y_0, Y_1\rangle_{\theta_{*}} = Y_0^{1-\theta_0} Y_1^{\theta_{*}}.
\]
Thus, by Theorem \ref{coneemb}, we deduce that
\[
T\colon \langle X_0, X_1\rangle_{\theta} \to \langle Y_0, Y_1\rangle_{\theta}
\]
is an order isomorphism for all $\theta \in (0, 1)$ and so Theorem \ref{realstable} applies.
\end{proof}

\vspace{3 mm}

\noindent Department of Mathematics (MAI)\newline
Link\"{o}ping University, Sweden \newline
E-mail: \texttt{irina.asekritova@liu.se} \newline

\noindent Department of Mathematics (MAI)\newline
Link\"{o}ping University, Sweden \newline
E-mail: \texttt{natan.kruglyak@liu.se}\newline

\noindent Faculty of Mathematics and Computer Science\newline
Adam~Mickiewicz University in Pozna\'n \newline
Uniwersytetu Pozna{\'n}skiego 4\newline
61-614 Pozna{\'n}, Poland\newline
E-mail: \texttt{mastylo$@$math.amu.edu.pl}

\begin{thebibliography}{99}

\bibitem{AM} E.~Albrecht and V.~M\"uller, \emph{Spectrum of interpolated operators},
Proc. Amer. Math. Soc. \textbf{129}~(2001), no.~3, 807–-814.

\bibitem{AB}
C.~Aliprantis and O.~Burkinshaw, \emph{Positive operators}, Pure and Applied Mathematics, 119.
Academic Press, Inc., Orlando, FL, 1985.

\bibitem{BL}
J.~Bergh and J.~L\"ofstr\"om, \emph{Interpolation spaces.An Introduction}, Springer, Berlin 1976.

\bibitem{BK}
Y.~Brudnyi and N.~Kruglyak, \emph{Interpolation functors and interpolation spaces}, Volume 1
North-Holland, Amsterdam 1991.

\bibitem{Ca1}
A.\,P.~Calder\'on, \emph{Intermediate spaces and interpolation, the complex method}, Studia Math.
\textbf{24}~(1964), 113–-190.

\bibitem{Ca2}
A.\,P.~Calder\'on, \emph{Boundary value problems for the Laplace equation in Lipschitzian domains,}
Recent progress in Fourier analysis (El Escorial, 1983), 33-–48, North-Holland Math. Stud., 111,
Notas Mat., 101, North-Holland, Amsterdam, 1985.

\bibitem{Cwikel}
M.~Cwikel, \emph{Complex interpolation spaces, a discrete definition and reiteration}, Indiana Univ. Math. J.
\textbf{27}~(1978), no.~6,
1005-–1009.

\bibitem{CKMR}
M.~Cwikel, N.~Kalton, M.~Milman and R.~Rochberg, \emph{A unified theory of commutator estimates for a~class
of interpolation methods}, Adv. Math. \textbf{169}~(2002), no.~2, 241–-312.

\bibitem{CN}
M.~Cwikel and Per G.~Nilsson, \emph{Interpolation of weighted Banach lattices}, Memoirs Amer.
Math. Soc. Vol. 165, No 785, 2003.

\bibitem{GP} J.~Gustavsson and J.~Peetre, \emph{Interpolation of Orlicz
spaces}, Studia Math. \textbf{60}~(1977), 33-–59.


\bibitem{Ivtsan}
A.~Ivtsan, \emph{Stafney's lemma holds for several "classical'' interpolation methods}, Proc. Amer.
Math. Soc. \textbf{140}~(2012), no.~3, 881-–889.

\bibitem{Janson}
S.~Janson, \emph{Minimal and maximal methods of interpolation},
J. Funct. Anal. \textbf{44}~(1981), no.~1, 50-–73.

\bibitem{KMM} N.~Kalton, S.~Mayboroda and M.~Mitrea, \emph{Interpolation of Hardy-Sobolev-Besov-Triebel-Lizorkin
spaces and applications to problems in partial differential equations}, Interpolation theory and applications,
121-–177, Contemp. Math., 445, Amer. Math. Soc., Providence, RI, 2007

\bibitem{Kato}
T.~Kato, \emph{Perturbation theory for nullity, deficiency and other quantities
of linear operators}, J. Anal. Math. \textbf{6}~(1958), 261--322.

\bibitem{KM}
N.~Krugljak and M.~Milman, \emph{A distance between orbits that controls commutator
estimates and invertibility of of operators}, Advances in Math. \textbf{182}~(2004),
78-–123.

\bibitem{LT}
J.~Lindenstrauss and L.~Tzafriri, \emph{Classical Banach spaces II}, Springer-Verlag, Berlin, 1979.

\bibitem{Lions}
J.\,L.~Lions, \emph{Sur les espaces d'interpolation; dualit\'e}, Math. Scand. \textbf{9}~(1961), 147--177.

\bibitem{LP}
J.\,L.~Lions and J.~Peetre, \emph{Sur une classe d'espaces d'interpolation}, Inst. Hautes Etudes Sci.
Publ. Math. \textbf{19}~(1964), 5-–68.

\bibitem{Lo78}
G.\,Ya.~Lozanovski{\u\i}, \emph{Transformations of ideal Banach spaces by means of concave functions},
In: Qualitative and approximate methods for investigation of operator equations, Jaroslavl' 1978,
122--148 (Russian).


\bibitem{Mi} M.~Milman, private communication.

\bibitem{Nilsson}
P.~Nilsson, \emph{Interpolation of Banach lattices}, Studia Math. \textbf{82}~(1985),
no.~2, 135-–154.

\bibitem{Peetre}
J.~Peetre, \emph{Sur l'utilization des suites inconditionellement sommables dans la
th\'eorie des espaces d'interpolation}, Rend. Sem. Mat. Univ. Padova \textbf{46}~(1971),
173-–190.

\bibitem{PV}
J.~Pipher and G.~Verchota, \emph{The Dirichlet problem in $L^p$ for the biharmonic equation
on Lipschitz domains}, Amer. J. Math. \textbf{114}~(1992), no.~5, 923-–972.

\bibitem{Ransford}
T.\,J.~Ransford, \emph{The spectrum of an interpolated operator and analytic multivalued
functions}, Pacific J. Math. \textbf{121}~(1986), no.~2, 445-–466.


\bibitem{Sh}
I.\,Ja~Shneiberg, \emph{Spectral properties of linear operators in interpolation families of
Banach spaces}, Math. Issled. \textbf{9}~(1974), 214--229 (Russian).


\bibitem{Zafran}
M.~Zafran, \emph{Spectral theory and interpolation of operators}, J. Funct. Anal. \textbf{36}~(1980),
no.~2, 185-–204.
\end{thebibliography}
\end{document}